\documentclass[a4paper,12pt]{article}

\usepackage{amsmath,amssymb,amsfonts,amsthm}
\usepackage[authoryear]{natbib}
\usepackage{graphicx}
\usepackage{float}
\usepackage{url}

\theoremstyle{plain}
\newtheorem{thm}{Theorem}[section]
\newtheorem{lem}[thm]{Lemma}
\newtheorem{prop}[thm]{Proposition}
\newtheorem{rem}[thm]{Remark}
\newtheorem{assumption}[thm]{Assumption}

\newcommand{\C}{\mathbb{C}}
\newcommand{\R}{\mathbb{R}}
\renewcommand{\S}{\mathbb{S}}
\newcommand{\w}{\omega}

\newcommand{\dd}{\mathrm{d}}

\newcommand{\len}{\mathrm{len}}

\newcommand{\Var}{\mathrm{Var}}
\newcommand{\Vol}{\mathrm{Vol}}

\begin{document}

\title{Optimal experimental design that minimizes the width of simultaneous conf
idence bands}

\author{Satoshi Kuriki\thanks{The Institute of Statistical Mathematics, 10-3 Mid
oricho, Tachikawa, Tokyo 190-8562, Japan, Email: {\tt kuriki@ism.ac.jp}}
 \ \ and\ \ %
Henry P. Wynn\thanks{The London School of Economics and Political Science, 
Houghton Street, London WC2A 2AE, UK, Email: {\tt h.wynn@lse.ac.uk}}}

\date{}

\maketitle

\begin{abstract}
We propose an optimal experimental design for a curvilinear regression model that minimizes the band-width of simultaneous confidence bands.
Simultaneous confidence bands for curvilinear regression are constructed by evaluating the volume of a tube about a curve that is defined as a trajectory of a regression basis vector (Naiman, 1986). 
The proposed criterion is constructed based on the volume of a tube, and the corresponding optimal design that minimizes the volume of tube is referred to as the tube-volume optimal (TV-optimal) design.
For Fourier and weighted polynomial regressions, the problem is formalized as one of minimization over the cone of Hankel positive definite matrices, and the criterion to minimize is expressed as an elliptic integral.
We show that the M\"obius group keeps our problem invariant, and hence, minimization can be conducted over cross-sections of orbits.
We demonstrate that for the weighted polynomial regression and the Fourier regression with three bases, the tube-volume optimal design forms an orbit of the M\"obius group containing D-optimal designs as representative elements.

\end{abstract}

\section{Introduction}
\label{sec:introduction}

Suppose that we observe pairs of explanatory variables $x_i\in\mathcal{X}$ and response variables $y_i\in\R$, $i=1,\ldots,N$.
Here, $\mathcal{X}\subset\R$ is a domain of explanatory variables, and typically, a segment of $\mathbb{R}$.
For such data, we consider the regression model
\[
 y_i = b^\top f(x_i) + \varepsilon_i,\quad \varepsilon_i \sim N\bigl(0,\sigma^2(x)\bigr)\ \ \mbox{i.i.d.},
\]
where $b=(b_1,\ldots,b_n)^\top$ is an unknown coefficient vector, and
$f(x)=(f_1(x),\ldots,f_n(x))^\top$, $x\in\mathcal{X}$, is a piecewise smooth regression basis vector.
For the problem of this paper, we assume that the variance function $\sigma^2(x)>0$ is known.
When $\sigma^2(x)$ is not a constant, the regression model is called {\it weighted\/}.

For experimental design, it is assumed that the explanatory variables $x_i$ can be chosen arbitrarily within its domain $\mathcal{X}\subset\mathbb{R}$.
The allocation of $\{x_1,\ldots,x_N\}\subset\mathcal{X}$ to optimize some target function is called optimal experimental design.
For example, for D-optimality, we take a function $\det(\Sigma)$ with $\Sigma=\Var(\widehat b)$, where $\widehat b$ is the ordinary least square (OLS) estimator of $b$.
Here,
\[
 \Sigma = M^{-1}\quad\mbox{with}\quad M=\sum_{i=1}^N f(x_i) f(x_i)^\top \frac{1}{\sigma^2(x_i)},
\]
the information matrix.

Following \cite{Kiefer-Wolfowitz59}, in an optimal design, allocation $\{x_1,\ldots,x_N\}$ is regarded as the probability measure over $\mathcal{X}$ with mass $p_i=1/N$ at each point $x_i$.
We write this discrete probability measure as
\[
 \left\{\begin{matrix} x_i \\ p_i \end{matrix}\right\}_{1\le i\le N} =
 \left\{\begin{matrix} x_1 & \cdots & x_N \\ \frac{1}{N} & \cdots & \frac{1}{N} \end{matrix}\right\}.
\]
Viewed from this point, the problem is formalized as that of optimization with respect to the probability measure over $\mathcal{X}$.
Most criteria in the literature including the criterion mentioned above are convex or concave functionals of the probability measure, and can be considered in the framework of convex analysis \citep{Wynn85,Pukelsheim06}.

In this paper, we propose a new non-convex criterion based on simultaneous confidence bands.
The pointwise confidence band is based on the confidence region for regressor $b^\top f(x)$ at a fixed point $x$.
On the other hand, the simultaneous confidence band is the confidence region for the full regression curve $\{(x,b^\top f(x)) \mid x\in\mathcal{X}\} \subset\R^{2}$.
The standard form of the simultaneous confidence band of hyperbolic-type is of the form
\begin{equation}
\label{hyperbolic}
 b^\top f(x) \in \widehat b^\top f(x) \pm c_\alpha \sqrt{f(x)^\top\Sigma f(x)},
\end{equation}
where $u\pm v$ stands for the region $(u-v,u+v)$.
The threshold $c_\alpha$ is determined so that the event (\ref{hyperbolic}) holds for all $x\in\mathcal{X}$ with given probability $1-\alpha$
\citep{Working-Hotelling29,Scheffe59,Wynn-Bloomfield71,Liu10}.
The simultaneous confidence bands are useful when $x$ cannot be determined in advance.
(See \cite{vanDyk14} for an application in experimental particle physics.)
As shown in the next section, $c_\alpha$ is also a functional of the allocation $\{x_1,\ldots,x_N\}$, and hence, we can consider an optimal design that in some way minimizes both the threshold $c_\alpha$ and $f(x)^\top\Sigma f(x)$.
In fact, from the general equivalence theorem of \cite{Kiefer-Wolfowitz59}, the design measure that minimizes $\max_{x\in\mathcal{X}}f(x)^\top\Sigma f(x)$
coincides with the D-optimal design.
Therefore, we propose the use of $c_\alpha$ as a criterion of optimal design, and consider the corresponding optimal design as the tube-volume optimal (TV-optimal) design.
If a design is optimal under both the tube-volume criterion and the D-criterion, it becomes the universal optimal design to minimize the width of confidence bands (\ref{hyperbolic}).

From its definition, $c_\alpha$ is a complicated function of $\Sigma$.
However, when $\alpha$ is small, $c_\alpha$ tends to a simpler function.
This approximation is due to the volume-of-tube method used to construct simultaneous confidence bands in curvilinear regression curves
\citep{Naiman86,Johansen-Johnstone90,Sun-Loader94,Lu-Kuriki17}.
The volume-of-tube method is a methodology to approximate the probability of the maximum of a Gaussian random field
\citep{Sun93,Kuriki-Takemura01,Kuriki-Takemura09,Takemura-Kuriki02,Adler-Taylor07}.
As shown later, $c_\alpha$ corresponds to the upper tail probability of the maximum of a Gaussian field, and hence, the volume-of-tube method works well.

As concrete regression models, weighted polynomial and Fourier regressions are mainly covered here.
In these models, we will see that there is a group referred to as the M\"obius transform that keeps the tube-volume optimal design problem invariant.
In general, a group action simplifies problems.
(See Section 13 of \cite{Pukelsheim06} for invariant optimal experimental design.)
The use of such group invariance is another subject of this paper.

The optimal design problem focusing on the width of the simultaneous confidence bands can be formalized in a different way.
When comparing two estimated regression curves, \cite{Dette-Schorning16} and \cite{Dette-etal17} proposed to minimize the $L_p$- or $L_\infty$-norm of the variance function of the estimator of the difference between the two curves. 
They demonstrated that their proposal reduces the width substantially compared with the pair of optimized designs for individual regression models. 
Different from them, our objective function is the width $c_\alpha$ of the simultaneous confidence band standardized by standard deviation.

The outline of this paper is as follows.
Section \ref{sec:min-vol-design} summarizes the volume-of-tube formula to construct approximate simultaneous confidence bands,
and formalizes the tube-volume criterion and the corresponding optimal design.
Section \ref{sec:fourier} analyzes the tube-volume optimal designs for Fourier and weighted polynomial regressions.
The M\"obius group is proved to keep the optimization problem invariant, and hence can be used to reduce the dimension of the problem.
Using this consideration, Section \ref{sec:n=3} identifies the tube-volume optimal design in the weighted polynomial regression and the Fourier regression when $n=3$.
Some proofs are given in Appendix.

\section{Tube-volume optimal design}
\label{sec:min-vol-design}

\subsection{Volume-of-tube formula for simultaneous confidence bands}
\label{subsec:volume}

In this subsection, we briefly summarize the volume-of-tube method.
This is a general methodology used to approximate the probability of the maximum of a smooth Gaussian random process or random field.
Here, we describe how this method is used to determine threshold $c_\alpha$.

As mentioned in Section \ref{sec:introduction}, threshold $c_\alpha$ should be determined as a solution $c=c_\alpha$ of
\begin{align}
 \Pr\biggl(\,\frac{|\,\widehat b^\top f(x)-b^\top f(x)|}
                  {\sqrt{f(x)^\top\Sigma f(x)}} <c,\,\forall x\in\mathcal{X}\biggr)
&= 1-
 \Pr\biggl(\max_{x\in\mathcal{X}}
 \frac{|\,(\widehat b-b)^\top f(x)|}{\Vert\Sigma^{\frac{1}{2}}f(x)\Vert}
 > c\biggr) \nonumber \\
&= 1-\alpha,
\label{calpha2}
\end{align}
where
$\Sigma^{\frac{1}{2}}$ is a symmetric square-root matrix such that
$(\Sigma^{\frac{1}{2}})^2=\Sigma$.

We define the normalized basis vector and its trajectory as
\begin{equation}
\label{psi-gamma}
 \psi_\Sigma(x) = \frac{\Sigma^{\frac{1}{2}}f(x)}{\Vert\Sigma^{\frac{1}{2}} f(x)\Vert}, \qquad
 \gamma_\Sigma = \bigl\{ \pm\psi_\Sigma(x) \mid x\in \mathcal{X} \bigr\},
\end{equation}
respectively.
From this definition, the trajectory is a subset of the ($n-1$)-dimensional unit sphere:
\[
 \gamma_\Sigma\subset\S^{n-1} = \{ u\in\R^n \mid \Vert u\Vert =1 \}.
\]
In particular, when $\mathcal{X}$ is a segment, $\gamma_\Sigma$ is a curve on the unit sphere.
Let $\Vol_1(\cdot)$ denote the one-dimensional volume, that is, the length.
Then, when $c$ is large, the volume-of-tube method provides an approximation to the upper tail probability of the maximum in (\ref{calpha2}).
Further, let $\chi^2_\nu$ denote the chi-square random variable with $\nu$ degrees of freedom.
\begin{prop}
\label{prop:naiman}
(i) As $c\to\infty$,
\begin{equation}
\label{tail}
 \Pr\biggl(\max_{x\in\mathcal{X}}\frac{|\,(\widehat b-b)^\top f(x)|}
              {\Vert \Sigma^{\frac{1}{2}}f(x)\Vert} >c\biggr) \sim
 \frac{\Vol_1(\gamma_\Sigma)}{2\pi} \Pr\bigl(\chi^2_{2}>c^2\bigr).
\end{equation}
(ii) For all $c>0$, the left-hand side of (\ref{tail}) is bounded above by
\begin{equation}
 \frac{\Vol_1(\gamma_\Sigma)}{2\pi} \Pr\bigl(\chi^2_2>c^2\bigr)
 + \chi(\gamma_\Sigma)\Pr\bigl(\chi^2_1>c^2\bigr),
\label{naiman}
\end{equation}
where $\chi(\gamma_\Sigma)$ is the number of the connected components of the set $\gamma_\Sigma\,(\subset\S^{n-1})$
provided that any connected component of $\gamma_\Sigma$ is not a closed curve.
\end{prop}

If we admit approximation (\ref{tail}), an approximate threshold $c_\alpha$ can be determined from the equation
\[
 \frac{\Vol_1(\gamma_\Sigma)}{2\pi} \Pr\bigl(\chi^2_{2}>c_\alpha^2\bigr) = \alpha.
\]
This means that the smaller the value of $\Vol_1(\gamma_\Sigma)$, the smaller is $c_\alpha$.

The statement (ii) above is due to \cite{Naiman86}.
Alternative proofs of the inequality can be found in \cite{Johnstone-Siegmund89} and \cite{Takemura-Kuriki02}.
See \cite{Lu-Kuriki17} for a generalization of Naiman's inequality.
By equating (\ref{naiman}) to be $\alpha$, we have a conservative threshold for the simultaneous confidence band.

The volume-of-tube method in Proposition \ref{prop:naiman} is based on the property that the normalized vector $\eta=\xi/\Vert\xi\Vert$, $\xi=\Sigma^{-\frac{1}{2}}(\widehat b-b)$, is distributed uniformly on the unit sphere $\S^{n-1}$, independently of its length $\Vert\xi\Vert$.
This occurs when the observation error $(\varepsilon_1,\ldots,\varepsilon_n)^\top$ is distributed as the elliptically contoured distribution, and hence the generalized least square estimator $\widehat b$ follows the elliptically contoured distribution as well (see, e.g., Theorem 2.6.3 of \cite{Fang-Zhang90}).
Part (ii) of Proposition \ref{prop:naiman} still holds as follows.
\begin{prop}
\label{prop:elliptically}
Suppose that $\xi=(\xi_1,\ldots,\xi_n)^\top=\Sigma^{-\frac{1}{2}}(\widehat b-b)$ is distributed according to an elliptically contoured distribution with mean zero and an identity covariance matrix.
Then, for all $c>0$,
\begin{equation}
 \Pr\biggl(\max_{x\in\mathcal{X}}\frac{|\,(\widehat b-b)^\top f(x)|}
              {\Vert \Sigma^{\frac{1}{2}}f(x)\Vert} >c\biggr) \le
 \frac{\Vol_1(\gamma_\Sigma)}{2\pi} \Pr\bigl(R_2^2>c^2\bigr)
 + \chi(\gamma_\Sigma)\Pr\bigl(R_1^2>c^2\bigr),
\label{naiman_elliptically}
\end{equation}
where $R_k^2=\sum_{i=1}^k\xi_i^2$, provided that any connected component of $\gamma_\Sigma$ is not a closed curve.
\end{prop}
The proof is given in Appendix \ref{subsec:elliptically}.
Part (i) of Proposition \ref{prop:naiman} does not hold in general.
It depends on the tail behavior of the elliptically contoured distribution.

\subsection{Tube-volume criterion}

From (\ref{naiman}), we find that the smaller the value of $\Vol_1(\gamma_\Sigma)$, the narrower is the width of the confidence band.
In this subsection, we formalize the experimental design optimization problem of the allocation of explanatory variables to minimize $\Vol_1(\gamma_\Sigma)$.

Here, we give our assumptions on $f(x)$.
\begin{assumption}
$f:\mathcal{X}\to\R^n$ is a continuous and piecewise $C^1$-function.
Image $f(\mathcal{X})$ spans $\R^n$.
\end{assumption}

From elementary geometry, the volume of $\gamma_\Sigma$ in (\ref{psi-gamma}) is given by
\begin{align}
 \Vol_1(\gamma_\Sigma)
&=
 2 \int_{\mathcal{X}} \biggl\Vert \frac{\dd \psi_\Sigma(x)}{\dd x}\biggr\Vert \, \dd x \nonumber \\
&=
 2 \int_{\mathcal{X}} \biggl\Vert \frac{\dd}{\dd x}\biggl(\frac{\Sigma^{\frac{1}{2}}f(x)}{\Vert \Sigma^{\frac{1}{2}}f(x) \Vert}\biggr) \biggr\Vert \, \dd x \nonumber \\
&=
 2 \int_{\mathcal{X}}
 \frac{\sqrt{(f(x)^\top \Sigma f(x))(g(x)^\top \Sigma g(x)) - (f(x)^\top \Sigma g(x))^2}}{f(x)^\top \Sigma f(x) }\, \dd x \nonumber \\
&=
 2 \int_{\mathcal{X}} \frac{\det\left(
 \begin{pmatrix} f(x)^\top \\ g(x)^\top \end{pmatrix} \Sigma
 \bigl(f(x), g(x)\bigr)
 \right)^{\frac{1}{2}}_{2\times 2}}
 {f(x)^\top \Sigma f(x)} \, \dd x,
\label{len}
\end{align}
where $g(x)=\dd f(x)/\dd x$.
We call (\ref{len}) the tube-volume (TV) criterion.

Note that (\ref{len}) is invariant with respect to scale
$\Sigma\mapsto k\Sigma$ ($k>0$).
Because of this,
we introduce the set of all nonnegative finite measures
 (not necessarily probability measures)
on $\mathcal{X}$ denoted by $\mathcal{P}$.
Each element of $\rho\in\mathcal{P}$ corresponds to an experimental design.
This is an extension of the design measure of \cite{Kiefer-Wolfowitz59}.
We also extend the set of moment matrices.
Thus,
the set of all non-singular information matrices is denoted by
\begin{align*}
 \mathcal{M}
 =& \biggl\{ \int_{\mathcal{X}} f(x) f(x)^\top \frac{1}{\sigma^2(x)}\dd \rho(x) \succ 0 \mid \rho\in\mathcal{P} \biggr\} \\
 =& \biggl\{ \int_{\mathcal{X}} f(x) f(x)^\top \dd \rho(x) \succ 0 \mid \rho\in\mathcal{P} \biggr\},
\end{align*}
where ``$\succ 0$'' denotes positive definiteness.
From its definition, $\mathcal{M}$ forms a convex cone.
Our optimal design problem is formulated as minimizing $\Vol_1(\gamma_\Sigma)$ in (\ref{len}) subject to $\Sigma^{-1}\in\mathcal{M}$.
The lemma below is a direct consequence of this invariance.
\begin{lem}
\label{lem:proportional}
The design $\left\{\begin{matrix} x_i \\ p_i \end{matrix}\right\}_{1\le i\le N}$ with variance function $\sigma_1^2(x)$, and the design $\left\{\begin{matrix} x_i \\ q_i \end{matrix}\right\}_{1\le i\le N}$ with variance function $\sigma_2^2(x)$,
$q_i=k p_i \sigma_2^2(x_i)/\sigma_1^2(x_i)$, give the same volume,
where $k>0$ is a constant so that $\sum_i q_i=1$.
\end{lem}
\begin{proof}
The information matrices $M_1$ and $M_2$ of the two designs satisfy $M_1=k M_2$.
\end{proof}

A difficulty with this optimization problem is that this is not a convex problem.
Figure \ref{fig:nonconvexity} depicts the volume $\Vol_1(\gamma_{M(c)^{-1}})$ for a mixing design connecting two Fourier designs
with three element bases ((\ref{fourier}) with $n=3$)
with weights $1-c$ and $c$\,:
\[
 \left\{\begin{matrix} t_i \\ p_i \end{matrix}\right\}_{i=1,2,3}
 = \left\{\begin{matrix} -\frac{1}{3} & 0 & \frac{1}{3} \\[1mm] \frac{1}{3} & \frac{1}{3} & \frac{1}{3} \end{matrix}\right\}
\ \ \mbox{and} \ \ %
 \left\{\begin{matrix} -\frac{1}{12} & 0 & \frac{1}{12} \\[1mm] \frac{1}{3} & \frac{1}{3} & \frac{1}{3} \end{matrix}\right\}.
\]
The information matrix is
\[
M(c) = (1-c)
\begin{pmatrix}
 1 & 0 & 0 \\
 0 & 1 & 0 \\
 0 & 0 & 1 \\
\end{pmatrix}
+c
\begin{pmatrix}
 1 & 0 & \frac{\sqrt{2}+\sqrt{6}}{3} \\
 0 & \frac{1}{3} & 0 \\
 \frac{\sqrt{2}+\sqrt{6}}{3} & 0 & \frac{5}{3}
\end{pmatrix}.
\]
We see that $\Vol_1(\gamma_{M^{-1}(c)})$ is not convex in $c$.

\begin{figure}[h]
\begin{center}
\scalebox{0.6}{\includegraphics{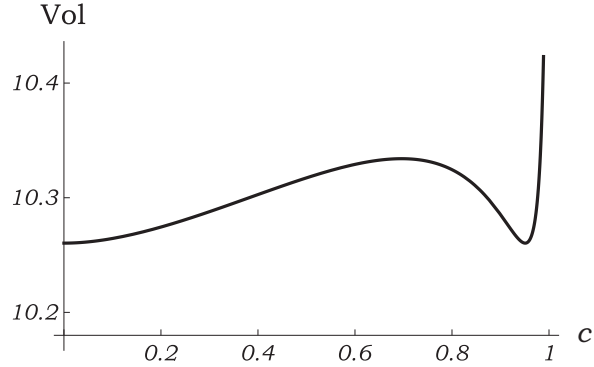}}
\end{center}
\caption{$\Vol_1(\gamma_{M(c)^{-1}})$ for a mixing design.}
\label{fig:nonconvexity}
\end{figure}

\section{Tube-volume optimal design for polynomial and Fourier regressions}
\label{sec:fourier}

\subsection{Equivalence between weighted polynomial regression and Fourier regression}

From this section, we focus on Fourier regression and weighted polynomial regression.
In linear optimal design theory, the Fourier regression model has been used as one of the standard models to see the performance of the proposed criteria.
The weighted polynomial introduced here has the same mathematical structure as the Fourier regression and will be used to analyze it.

The Fourier (trigonometric) regression has basis vector
\begin{align}
f(x)
=& f_F(x) \nonumber \\
=&
\begin{cases}
 \bigl(1,\sqrt{2}\sin(2\pi x),\sqrt{2}\cos(2\pi x),\sqrt{2}\sin(4\pi x),\ldots,\sqrt{2}\cos(2\pi m x)\bigr)^\top \\ \hspace*{85mm} (n=2m+1), \\
 \bigl(\sqrt{2}\cos(\pi x),\sqrt{2}\sin(\pi x),\sqrt{2}\cos(3\pi x),\ldots,\sqrt{2}\cos(\pi (n-1) x)\bigr)^\top \\ \hspace*{85mm} (n=2m)
\end{cases}
\label{fourier}
\end{align}
defined on the domain $\mathcal{X}=(-1/2,1/2]$.
For the Fourier regression, we only deal with the constant variance $\sigma^2(x)=\sigma_F^2(x)\equiv 1$.
Although the Fourier regression is not used for even values of $n$ in practice, we define it for the sake of consistency.

The polynomial regression is a regression model with basis vector
\begin{equation}
\label{polynomial}
 f(x)=f_P(x)=(1,x,x^2,\ldots,x^{n-1})^\top \in\R[x]^n.
\end{equation}
Here, we set the domain $\mathcal{X}$ to be the whole real line $\mathbb{R}$.
The case where $\mathcal{X}$ is an infinite interval will be briefly discussed in Section \ref{sec:summary}.
For the polynomial regression, we assume the variance function of form $\sigma^2(x) = Q(x)^{n-1}$, where $Q(x)$ is an arbitrary positive quadratic function.
As a canonical form of this class of variance functions, we use
\begin{equation}
\label{sigmaP}
 \sigma_P^2(x) = (1+x^2)^{n-1}.
\end{equation}
Later, we introduce a parameterization for $Q(x)$ (see (\ref{sigmaPa})).

In this subsection, we see that under the tube-volume criterion,
the optimization problem for the Fourier regression is equivalent to that for the weighted polynomial regression.
That is, the optimization problem in the Fourier regression can be translated to one in the weighted polynomial regression, and vice versa.

The model we discuss is a special case of the model proposed by \cite{Dette-etal99}, Section 2.2, who study the D-optimality.
Further, \cite{Dette-Melas03} make use of the connection between the weighted polynomial and Fourier regressions.
We will return briefly to question on the D-optimality in Section \ref{subsec:D-optimal} below.

From the lemma below, the transformation $x=\tan(\pi t)$ connects the two regression models.
\begin{lem}
\label{lem:B}
There exists an $n\times n$ non-singular matrix $B$ such that for all $x\in\R$ and $t\in(-1/2,1/2]$ satisfying $x=\tan(\pi t)$, we have
\begin{equation}
\label{B}
 f_F(t) = B f_P(x) \lambda_0(x), \quad \lambda_0(x)=1/(1+x^2)^{(n-1)/2}.
\end{equation}
\end{lem}
\begin{proof}
Note that
\begin{equation}
\label{high-school}
 \sin(2\pi t) = \frac{2x}{1+x^2}, \quad
 \cos(2\pi t) = \frac{1-x^2}{1+x^2}, \quad
 \dd t = \frac{\dd x}{\pi(1+x^2)}.
\end{equation}
It is known that
\begin{align*}
 \sin(2\pi k t) =& \sum_{r=0}^{[(k-1)/2]} (-1)^r {k \choose 2r+1} \sin^{2r+1}(2\pi t)\cos^{k-2r-1}(2\pi t), \\
\cos(2\pi k t) =& \sum_{r=0}^{[k/2]} (-1)^r {k \choose 2r} \sin^{2r}(2\pi t)\cos^{k-2r}(2\pi t)
\end{align*}
\cite[pages 186--187]{Moriguti-etal57}.
Substituting the formulas for $\sin(2\pi t)$ and $\cos(2\pi t)$ in (\ref{high-school}) and expressing $f_F(t)$ as a rational function in $x$, we have formula (\ref{B}).

To prove that $B$ is non-singular, consider the integral
\begin{equation}
\label{parentheses}
 \int_{-1/2}^{1/2}  f_F(t) f_F(t)^\top \dd t
= B \biggl( \int_{-\infty}^{\infty} f_P(x) \lambda_0(x)^2 f_P(x)^\top \frac{\dd x}{\pi(1+x^2)} \biggr) B^\top.
\end{equation}
Here, we used (\ref{high-school}).
The left-hand side is the identity matrix $I_n$ by standard orthogonality.
Hence, it is enough to check that the integrals in the parentheses of the right-hand side exist.
The matrix in the parentheses of the right-hand side of (\ref{parentheses}) is $(B^\top B)^{-1}$ with $(i,j)$ element
\begin{equation}
\label{BB}
 (B^\top B)^{-1}_{i,j} = \int_{-\infty}^{\infty}\frac{x^{i+j-2}}{\pi(1+x^2)^n} \dd x = \begin{cases} \displaystyle
 \frac{\Gamma\bigl(\frac{i+j-1}{2}\bigr)\Gamma\bigl(n-\frac{i+j-1}{2}\bigr)}{\pi\Gamma(n)} & (\mbox{$i+j$ is even}), \\ 0 &  (\mbox{$i+j$ is odd}), \end{cases}
\end{equation}
which exists for $i,j\le n$.
Hence, $B$ is non-singular.
\end{proof}
Lemma \ref{lem:B} means that the Fourier regression model
$y_i=b^\top f_F(t_i)+\varepsilon_i$, $\varepsilon_i\sim N(0,1)$,
is rewritten as the weighted polynomial model
$\widetilde y_i=\widetilde b^\top f_P(x_i)+\widetilde\varepsilon_i$,
by letting $x_i=\tan(\pi t_i)$, $\widetilde y_i=\lambda_0(x_i)^{-1}y_i$,
$\widetilde b=B^\top b$, and
$\widetilde\varepsilon_i\sim N(0,\lambda_0(x_i)^{-2})$.

When $n=3, 4$,
\[
 B=
\begin{pmatrix}
 1 & 0 & 1 \\ 0 & 2\sqrt{2} & 0 \\ \sqrt{2} & 0 & -\sqrt{2}
\end{pmatrix}, \ \ %
\begin{pmatrix}
 1 & 0 & 1 & 0 \\ 0 & 1 & 0 & 1 \\ 1 & 0 & -3 & 0 \\ 0 & 3 & 0 & -1
\end{pmatrix},
\]
respectively.

The set of information matrices for the polynomial regression
\begin{equation}
 \mathcal{M}_P
 = \biggl\{ \int_{-\infty}^\infty f_P(x) f_P(x)^\top \dd \rho(x) \succ 0 \mid \rho\in\mathcal{P} \biggr\}
\label{MP}
\end{equation}
is referred to as the moment cone \citep{Karlin-Studden66}.
The set of information matrices for the Fourier regression is given by
\begin{align}
 \mathcal{M}_F
& = \biggl\{ \int_{-1/2}^{1/2} f_F(t) f_F(t)^\top \dd \rho(t) \succ 0 \mid \rho\in\mathcal{P} \biggr\} \nonumber \\
& = \bigl\{ B M B^\top \mid M\in \mathcal{M}_P \bigr\} = B \mathcal{M}_P B^\top.
\label{MF}
\end{align}

The following lemma gives the equivalence of the Fourier regression and the polynomial regression as the optimization problem for the tube-volume criterion.

\begin{thm}
\label{thm:VolF-VolP}
Let $\Vol_F(\gamma_{M^{-1}})$ and $\Vol_P(\gamma_{M^{-1}})$ be the length of $\gamma_{M^{-1}}$ given in (\ref{len}) with $f(x)$ being $f_F(x)$ in (\ref{fourier}), and $f_P(x)$ in (\ref{polynomial}), respectively.
Then, it holds that $\Vol_P(\gamma_{M^{-1}})=\Vol_F(\gamma_{(B M B^\top)^{-1}})$.
\end{thm}

\begin{proof}
The derivatives of $f_F(t)$ and $f_P(x)$ are denoted by
$g_F(t)=\dd f_F(t)/\dd t$ and $g_P(x)=\dd f_P(x)/\dd x$, respectively.
Then,
\[
 \bigl(f_F(t),g_F(t)\bigr) = B \bigl(f_P(x),g_P(x)\bigr) \begin{pmatrix} \lambda_0(x) & \dot\lambda_0(x) \\ 0 & \lambda_0(x) \end{pmatrix}
 \begin{pmatrix} 1 & 0 \\ 0 & \frac{\dd x}{\dd t} \end{pmatrix}.
\]
Therefore,
\[
 f_F(t)^\top (B M B^\top)^{-1} f_F(t)
= f_P(x)^\top M^{-1} f_P(x) \times \lambda_0(x)^2,
\]
\begin{align*}
& \det\left(
 \begin{pmatrix} f_F(t)^\top \\ g_F(t)^\top \end{pmatrix} (B M B^\top)^{-1}
 \bigl(f_F(t),g_F(t)\bigr) \right) \\
& \quad =
 \det\left(
 \begin{pmatrix} f_P(x)^\top \\ g_P(x)^\top \end{pmatrix} M^{-1}
 \bigl(f_P(x),g_P(x)\bigr) \right)
 \times \lambda_0(x)^4 \Bigl(\frac{\dd x}{\dd t}\Bigr)^2,
\end{align*}
and
\begin{align*}
\int_{-1/2}^{1/2}
 \frac{\det\left(
 \begin{pmatrix} f_F(t)^\top \\ g_F(t)^\top \end{pmatrix} (B M B^\top)^{-1}
 \bigl(f_F(t),g_F(t)\bigr) \right)^{\frac{1}{2}}}
 {f_F(t)^\top (B M B^\top)^{-1} f_F(t)} \,\dd t \\
=
 \int_{-\infty}^\infty
 \frac{\det\left(
 \begin{pmatrix} f_P(x)^\top \\ g_P(x)^\top \end{pmatrix} M^{-1}
 \bigl(f_P(x),g_P(x)\bigr) \right)^{\frac{1}{2}}}
 {f_P(x)^\top M^{-1} f_P(x)} \,\dd x.
\end{align*}
\end{proof}

Theorem \ref{thm:VolF-VolP} and (\ref{MF}) imply that
\begin{eqnarray*}
&& \mbox{$M\in\mathcal{M}_P$ is the minimizer of $\Vol_P(M)$} \\
&\Leftrightarrow&
\mbox{$M'=B M B^\top\in\mathcal{M}_F$ is the minimizer of $\Vol_P(M')$.}
\end{eqnarray*}
That is, the optimization problems
for the polynomial regression and the Fourier regression are mathematically equivalent.
For example, the information matrix
$M=\sum_i f_F(t_i) f_F(t_i)^\top p_i\in \mathcal{M}_F$
for the Fourier regression, and the information matrix
for the polynomial regression
\begin{align*}
\mathcal{M}_P\ni
B^{-1} M (B^\top)^{-1}
=& \sum_i B^{-1} f_F(t_i) (B^{-1} f_F(t_i))^\top p_i \\
=& \sum_i f_P(x_i) f_P(x_i)^\top \lambda_0(x_i)^2 p_i, \quad x_i=\tan(\pi t_i), \\
=& \sum_i f_P(x_i) f_P(x_i)^\top \frac{1}{\sigma_P^2(x_i)} p_i
\end{align*}
give the same volume.

This equivalence is stated in terms of design measure as follows.
\begin{thm}
\label{thm:fourier-polynomial}
The design $\left\{\begin{matrix} t_i \\ p_i \end{matrix}\right\}_{1\le i\le N}$ for the Fourier regression, and the design
$\left\{\begin{matrix} x_i \\ p_i \end{matrix}\right\}_{1\le i\le N}$, $x_i=\tan(\pi t_i)$, for the weighted polynomial regression with variance function $\sigma_P^2(x)$ give the same volume.
If the former is tube-volume optimal in the Fourier regression, then so is the latter in the polynomial regression with variance $\sigma_P^2(x)$, and vice versa.
\end{thm}

In this paper, the (discrete) uniform designs in the Fourier regression
and their counterparts in the polynomial regression play important roles.
It is known that, in the Fourier regression, the uniform design in which $x_i$ are allocated as equally spaced with equal weights is D-optimal \citep{Guest58}.
Because of the symmetry, it is conjectured that the uniform design is the tube-volume optimal design as well.
In Section \ref{sec:n=3}, we prove that this is true for $n=3$, and conjecture that it is true for all $n$.

The $n$-point discrete uniform design for the Fourier regression symmetric about the origin is
\begin{equation}
\label{uniform}
 \left\{\begin{matrix} t^0_i \\ \frac{1}{n} \end{matrix}\right\}_{1\le i\le n},
 \quad
t^0_i= \frac{i}{n}-\frac{n+1}{2n}.
\end{equation}
For later use, we provide the concrete forms of the information matrix $M$ for the weighted polynomial designs with $\sigma^2(x)=\sigma_P^2(x)$ in (\ref{sigmaP}),
\begin{equation}
\label{uniform-poly}
 \left\{\begin{matrix} x^0_i \\ \frac{1}{n} \end{matrix}\right\}_{1\le i\le n},
 \quad
 \mbox{$x^0_i=\tan(\pi t^0_i)$ with $t^0_i$ given in (\ref{uniform})}.
\end{equation}

\begin{lem}
\label{lem:M}
The information matrix
 $M=(M_{i,j})$ of the weighted polynomial design (\ref{uniform-poly}) scaled such that $M_{1,1}=1$ is given by
\[
 M = (M_{i,j}) = (m_{i+j-2})_{1\le i,j\le n}, \quad
 m_k = \begin{cases} \displaystyle
 \frac{\Gamma\bigl(\frac{k+1}{2}\bigr)\Gamma\bigl(n-\frac{k+1}{2}\bigr)}{\sqrt{\pi}\Gamma\bigl(n-\frac{1}{2}\bigr)} & (\mbox{$k$ is even}), \\
 0 & (\mbox{$k$ is odd}\/). \end{cases}
\]
\end{lem}

For the proof, see Appendix \ref{subsec:gould}.
When $n=3$ and $4$,
\begin{equation}
\label{M3}
 M =
\begin{pmatrix}
 1 & 0 & \frac{1}{3} \\
 0 & \frac{1}{3} & 0 \\
 \frac{1}{3} & 0 & 1
\end{pmatrix}, \ \ %
\begin{pmatrix}
 1 & 0 & \frac{1}{5} & 0 \\
 0 & \frac{1}{5} & 0 & \frac{1}{5} \\
 \frac{1}{5} & 0 & \frac{1}{5} & 0 \\
 0 & \frac{1}{5} & 0 & 1
\end{pmatrix},
\end{equation}
respectively.

The key transform connecting Fourier and polynomial regressions was the tangent transform $x=\tan(\pi t)$.
For the same purpose, generalized transforms
$x = q \tan(\pi(t - \theta)) + r$, $q\ne 0$, can be used.
This is a composite map of the tangent transform and the M\"obius transform to be discussed below.

\subsection{The M\"obius group action on the moment cone}
\label{subsec:mobius}

In this subsection, we introduce the M\"obius group (transformation) acting on the set of design measures and the set of information matrices in polynomial regression.
We will show that the M\"obius group action reduces the dimension of the minimization problem for the tube-volume criterion.
For a recent paper in which the M\"obius transformation acts on polynomials, see \cite{Mackey-etal15}.

The real M\"obius transformation is defined on the extended real numbers $\overline{\mathbb{R}}=\mathbb{R}\cup\{\pm\infty\}$ as follows:
\[
 x \mapsto \varphi(x;a,b,c,d) = \frac{a x +b}{c x +d} \quad (ad-bc \ne 0).
\]
Here, we assume that
\[
 \pm\infty \mapsto \frac{a}{c}, \quad -\frac{d}{c} \mapsto \pm\infty.
\]
This forms a group with product
\begin{equation}
\label{group_product}
\varphi(\cdot;a'a+b'c,a'b+b'd,c'a+d'c,c'b+d'd) = \varphi(\cdot;a',b',c',d')\circ \varphi(\cdot;a,b,c,d).
\end{equation}
The inverse is
$\varphi^{-1}(\cdot;a,b,c,d)=\varphi(\cdot;d,-b,-c,a)$.
The identity element is $e=\varphi(\cdot;a,0,0,a)$, $a\ne 0$.
This is a subgroup of the complex M\"obius group referred to as projective general linear group $\mathrm{PGL}(2,\C)$.

Now, let $f_P(x)=(1,x,\ldots,x^{n-1})^\top$ be the polynomial basis.
We define an $n\times n$ matrix $A=A(a,b,c,d)$ as
\begin{equation}
\label{A}
 f_P\bigl(\varphi(x;a,b,c,d)\bigr) = \lambda(x;a,b,c,d) A f_P(x), \quad
 \lambda(x;a,b,c,d) = \frac{1}{(c x+d)^{n-1}}.
\end{equation}
We write the factor $\lambda$ as $\lambda(x;a,b,c,d)$ instead of $\lambda(x;c,d)$ to clarify that this is an invariant function under the group action (\ref{group_product}) in the sense that
\begin{align}
& \lambda\bigl(\varphi(x;a,b,c,d);a',b',c',d'\bigr) \lambda(x;a,b,c,d) \nonumber \\
&\quad = \lambda(x;a'a+b'c,a'b+b'd,c'a+d'c,c'b+d'd).
\label{product-lambda}
\end{align}

\begin{prop}
The $(i,j)$ element of $A=A(a,b,c,d)$ is
\[
 (A)_{i,j} = \sum_{l=\max(0,i+j-n-1)}^{\min(i-1,j-1)}
 {i-1 \choose l}
 {n-i \choose j-1-l} a^l b^{i-1-l} c^{\,j-1-l} d^{\,n+1-i-j+l}.
\]
\end{prop}

The proof is straightforward and omitted.
When $n=3$ and $n=4$,
\begin{equation}
\label{A3}
 A(a,b,c,d) =
 \begin{pmatrix}
  d^2 & 2cd & c^2 \\ bd & bc+ad & ac \\ b^2 & 2ab & a^2 \end{pmatrix},
\ \ %
 \begin{pmatrix}
  d^3 & 3 c d^2 & 3 c^2 d & c^3 \\
  b d^2 & 2 b c d + a d^2 & b c^2 + 2 a c d & a c^2 \\
  b^2 d & b^2 c + 2 a b d & 2 a b c + a^2 d & a^2 c \\
  b^3 & 3 a b^2 & 3 a^2 b & a^3 \end{pmatrix},
\end{equation}
respectively.

The set
\[
 \mathcal{A}=\{ A(a,b,c,d) \mid ad-bc \ne 0 \}
\]
is a representation of general linear group $\mathrm{GL}(2,\R)$ and hence forms a group \citep{Gross-Holman80}.
The proof of the proposition below is straightforward and omitted.
\begin{prop}
Set $\mathcal{A}$ forms a matrix algebraic group.
The identity matrix is $A(1,0,0,\allowbreak 1)=(-1)^{n-1}A(-1,0,0,-1)=I_n$, and
the inverse of $A=A(a,b,c,d)$ is given by $(ad-bc)^{-(n-1)} A(d,-b,-c,a)$.
\end{prop}

\begin{prop}
For $A=A(a,b,c,d)\in \mathcal{A}$, $\det(A)=(ad-bc)^{n(n-1)/2}$.
\end{prop}

\begin{proof}
For different $x_1,\ldots,x_n$, we have
\[
 A (f_P(x_1),\ldots,f_P(x_n)) = (f_P(y_1),\ldots,f_P(y_n)) \mathrm{diag}((cx_i+d)^{n-1}), \quad y_i=\frac{a x_i+b}{c x_i+d}.
\]
By taking determinants,
\begin{equation}
\label{vandelmonde}
 \det(A) \times (-1)^n \prod_{1\le i<j\le n}(x_i-x_j) =
 (-1)^n \prod_{1\le i<j\le n}(y_i-y_j) \prod_{1\le i\le n} (cx_i+d)^{n-1}.
\end{equation}
Substituting
\begin{align*}
 \prod_{1\le i<j\le n}(y_i-y_j)
&= \prod_{1\le i<j\le n}\Bigl(\frac{ax_i+b}{cx_i+d}-\frac{ax_j+b}{cx_j+d}\Bigr)= \prod_{1\le i<j\le n}\frac{(ad-bc)(x_i-x_j)}{(cx_i+d)(cx_j+d)} \\
&= (ad-bc)^{n(n-1)/2} \frac{\prod_{1\le i<j\le n}(x_i-x_j)}{\prod_{1\le i\le n}(cx_i+d)^{n-1}}
\end{align*}
into (\ref{vandelmonde}) yields $\det(A)=(ad-bc)^{n(n-1)/2}$.
\end{proof}

\begin{prop}
\label{prop:decomposition}
The M\"obius group is reparameterized as
\begin{align*}
\{\varphi(\cdot;a,b,c,d) \mid ad-bc\ne 0 \}
= \bigl\{\varphi(\cdot;q,r,0,1)\circ\varphi(\cdot;s,-t,t,s) \mid q\ne 0,\ s^2+t^2=1 \bigr\} \\
\sqcup \bigl\{\varphi(\cdot;q,r,0,1)\circ\varphi(\cdot;-s,t,t,s) \mid q\ne 0,\ s^2+t^2=1 \bigr\},
\end{align*}
where $\sqcup$ is the disjoint union.
Group $\mathcal{A}$ is reparameterized as
\begin{align*}
\mathcal{A}=
& \bigl\{ k A(q,r,0,1) A(s,-t,t,s) \mid k>0,\ q\ne 0,\ s^2+t^2=1 \bigr\} \\
& \sqcup \bigl\{ k A(q,r,0,1) A(-s,t,t,s) \mid k>0,\ q\ne 0,\ s^2+t^2=1 \bigr\}.
\end{align*}
$\det A(q,r,0,1) = q^{n(n-1)/2}$ and $\det A(\pm s,\mp t,t,s) = \pm 1$ for $s^2+t^2=1$.
\end{prop}
\begin{proof}
We have the following relations:
\[
 \varphi(x;a,b,c,d)
 = \varphi\Bigl(\varphi(x;\pm d,\mp c,c,d); \pm\frac{ad - bc}{c^2 + d^2}, \frac{ac + bd}{c^2 + d^2}, 0, 1\Bigr)
\]
and
\[
A(a,b,c,d)
 = (c^2+d^2)^{(n-1)/2} A\Bigl(\pm\frac{ad-bc}{c^2+d^2},\frac{ac+bd}{c^2+d^2},0,1\Bigr) \frac{A(\pm d,\mp c,c,d)}{(c^2+d^2)^{(n-1)/2}}.
\]
The results in the proposition follow by letting
$k=(c^2+d^2)^{(n-1)/2}$,
\[
q=\frac{ad-bc}{c^2+d^2}, \quad r=\frac{ac+bd}{c^2+d^2}, \quad s=\frac{d}{\sqrt{c^2+d^2}},\quad t=\frac{c}{\sqrt{c^2+d^2}}.
\]
\end{proof}

The sets of transformations
$\{\varphi(\cdot;\pm s,\mp t,t,s) \mid s^2+t^2=1 \}$ and
$\{\varphi(\cdot;q,r,0,1) \mid q\ne 0 \}$
form subgroups of the M\"obius group,
which are isomorphic to the orthogonal group $\mathrm{O}(2,\R)$ and the affine group acting on $\R$, respectively.

\begin{thm}
\label{thm:action}
Let $A\in \mathcal{A}$ and $M\in\mathcal{M}_P$ be $n\times n$ matrices.
Then, $A M A^\top\in\mathcal{M}_P$.
Moreover,
\[
 A \mathcal{M}_P A^\top =
 \bigl\{ A M A^\top \mid M \in \mathcal{M}_P \bigr\} = \mathcal{M}_P.
\]
That is, group $\mathcal{A}$ acts on the moment cone $\mathcal{M}_P$.
\end{thm}
\begin{proof}
Suppose that $M=\sum_i f_P(x_i) f_P(x_i)^\top w_i$.
Let
$y_i = \varphi(x_i;a,b,c,d)$.
Note that $\lambda(y;d,-b,-c,a) \lambda(x;a,b,c,d)= (ad-bc)^{-(n-1)}$.
Then,
\begin{align*}
A M A^\top
&= \sum_i A f_P(x_i) (A f_P(x_i))^\top w_i \\
&= \sum_i f_P(y_i) f_P(y_i)^\top \lambda(x_i;a,b,c,d)^{-2} w_i \\
&= \sum_i f_P(y_i) f_P(y_i)^\top (ad-bc)^{2(n-1)} \lambda(y_i;d,-b,-c,a)^2 w_i \\
&= \sum_i f_P(y_i) f_P(y_i)^\top v_i \in \mathcal{M}_P,
\end{align*}
where
\begin{equation}
\label{vw}
v_i = \lambda(x_i;a,b,c,d)^{-2} w_i = (ad-bc)^{2(n-1)} \lambda(y_i;d-b,-c,a)^2 w_i.
\end{equation}
Therefore, $A \mathcal{M}_P A^\top \subset \mathcal{M}_P$.
Because $\mathcal{A}$ is a group, $A \mathcal{M}_P A^\top = \mathcal{M}_P$.
\end{proof}

The M\"obius group action on the polynomial basis $f(x)$ has been introduced by (\ref{A}).
Similarly, we define the M\"obius group action on the variance function $\sigma^2(x)=Q(x)^{n-1}$.
This provides a parameterization for the variance function.

Using $\sigma_P^2(x)=(1+x^2)^{n-1}$ in (\ref{sigmaP}), for $ad-bc\ne 0$, we define
\begin{equation}
\label{sigmaPa0}
\sigma_P^2(\varphi(x;a,b,c,d)) = \sigma_P^2(x;a,b,c,d) \lambda(x;a,b,c,d)^2,
\end{equation}
or
\begin{equation}
\label{sigmaPa}
 \sigma_P^2(x;a,b,c,d)
 = \bigl\{ (b^2+d^2) + 2(a b+c d)x + (a^2+c^2) x^2 \bigr\}^{n-1}.
\end{equation}
Note that
$\sigma_P^2(x)=\sigma_P^2(x;1,0,0,1)$.
This is always positive because of $ad-bc\ne 0$.
For
\[
 \varphi(\cdot;a',b',c',d') = \varphi(\cdot;a_0,b_0,c_0,d_0)\circ\varphi(\cdot;a,b,c,d),
\]
as well as (\ref{product-lambda}), we have
\begin{equation}
\label{sigma-mobius}
 \sigma_P^2(\varphi(x;a,b,c,d);a_0,b_0,c_0,d_0) =
 \sigma_P^2(x;a',b',c',d') \lambda(x;a,b,c,d)^2.
\end{equation}

The parameterization (\ref{sigmaPa}) with $(a,b,c,d)$ is redundant, since $Q(x)$ has only three parameters.
The lemma below shows that the stabilizer keeping the variance $\sigma_P^2(\cdot;a,b,c,d)$ invariant is the orthogonal subgroup with dimension one.

\begin{lem}
$\sigma_P^2(\cdot;a,b,c,d)=\sigma_P^2(\cdot;a',b',c',d')$ if and only if there exist $s,t$, $s^2+t^2=1$ such that
\[
 \varphi(\cdot;a',b',c',d')=\varphi(\cdot;\pm s,\mp t,t,s)\circ\varphi(\cdot;a,b,c,d).
\]
\end{lem}

\begin{proof}
By direct calculations,
\[
 \sigma_P^2(\cdot;a_0,b_0,c_0,d_0)=\sigma_P^2(\cdot;1,0,0,1) \Leftrightarrow (a_0,b_0,c_0,d_0)=(\pm s,\mp t,t,s).
\]
On the other hand,
``$\sigma_P^2(\cdot;a_0,b_0,c_0,d_0)=\sigma_P^2(\cdot;1,0,0,1)$'' $\Leftrightarrow$ ``$\sigma_P^2(\varphi(\cdot;a,b,c,d);a_0,b_0,c_0,d_0)\allowbreak =\sigma_P^2(\varphi(\cdot;a,b,c,d);1,0,0,1)$'' $\Leftrightarrow$
``$\sigma_P^2(x;a',b',c',d')=\sigma_P^2(x;a,b,c,d)$ with $\varphi(\cdot;a',b',c',d') = \varphi(\cdot;a_0,b_0,c_0,d_0)\circ\varphi(\cdot;a,b,c,d)$'' by (\ref{sigma-mobius}).
\end{proof}

\begin{rem}
When $n=2$, $|ad-bc|/\{\pi\sigma_P^2(x;a,b,c,d)\}$ is a probability density of Cauchy distribution family on $x\in\R$.
Noting that $\dot\varphi(x;a,b,c,d)=(ad-bc)\lambda(x;a,b,c,d)^2$, and that
$a'd'-b'c'=(ad-bc)(a_0 d_0-b_0 c_0)$,
(\ref{sigma-mobius}) reads
\[
 \frac{|a_0 d_0-b_0 c_0|}{\pi\sigma_P^2(\varphi(x;a,b,c,d);a_0,b_0,c_0,d_0)} |\dot\varphi(x;a,b,c,d)| \,\dd x =
 \frac{|a'd'-b'c'|}{\pi\sigma_P^2(x;a',b',c',d')} \dd x.
\]
This means that the Cauchy distribution family is closed under the M\"obius transform \citep{McCullagh96}.
See also \cite{Kato-McCullagh14} for Cauchy families in directional statistics.
\end{rem}

\subsection{Canonical parameterizations for information matrices}

As we have shown in Section \ref{subsec:volume}, the optimal design problem is optimization with respect to matrix $M$ over the set of information matrices $\mathcal{M}$.
Here, $M=\int_{\mathcal{X}} f(x) f(x)^\top \frac{1}{\sigma^2(x)} \dd \rho(x)\in\mathcal{M}$ and the design measure $\rho\in\mathcal{P}$ is one-to-many.
For the sake of optimization, we need to parameterize the set $\mathcal{M}$.

We first consider $\mathcal{M}_P$ in (\ref{MP}) for the polynomial regression, and then interpret the results in terms of $\mathcal{M}_F$ in (\ref{MF}) for the Fourier regression.

The structure of the moment cone $\mathcal{M}_P$ is well-studied in the context of the classical moment problem.
One canonical parameterization for $\mathcal{M}_P$ is given in Chapter II, Section 3 of \cite{Karlin-Studden66}.
The statement is summarized in Proposition 3.1 of \cite{Kato-Kuriki13}.
\begin{prop}
$M\in \mathcal{M}_P$ is uniquely represented with $2n-1$ parameters $(w_0,\ldots,w_{n-1},\allowbreak x_1,\ldots,x_{n-1})$ as
\begin{align}
\label{KS-MP}
 M = \sum_{i=1}^{n-1} w_i f_P(x_i) f_P(x_i)^\top + w_0 f_P(\pm\infty) f_P(\pm\infty)^\top, \ \ w_i>0, \\
 -\infty<x_1<\cdots<x_{n-1}<\infty,
\nonumber
\end{align}
where we let $f_P(\pm\infty)=(0,\ldots,0,1)^\top$.
\end{prop}
Note that $f_P(\pm\infty)=\lim_{x\to\pm\infty} f_P(x)/x^{n-1}=\lim_{x\to\pm\infty} f_P(x)/(1+x^2)^{(n-1)/2}$.

Let $x_0=\pm\infty$.
From the same argument of the proof of Theorem \ref{thm:action},
by considering the M\"obius transform $x_i\mapsto (a x_i +b)/(c x_i +d)$, $i=0,1,\ldots,n-1$,
we find that the fixed point $x_0=\pm\infty$ in (\ref{KS-MP}) can be moved to
an arbitrary point in $\mathbb{R}$.

\begin{thm}
Let $x_0\in\mathbb{R}\cup\{\pm\infty\}$ be fixed arbitrarily.
$M\in \mathcal{M}_P$ is uniquely represented with $2n-1$ parameters $(w_0,\ldots,w_{n-1},x_1,\ldots,x_{n-1})$ as
\begin{align*}
 M = \sum_{i=1}^{n-1} w_i f_P(x_i) f_P(x_i)^\top + w_0 f_P(x_0) f_P(x_0)^\top, \ \ w_i>0, \\
 x_i\ne x_0, \ \ -\infty<x_1<\cdots<x_{n-1}\le\infty,
\end{align*}
where we assume that $f_P(\pm\infty)=(0,\ldots,0,1)^\top$.
\end{thm}

The counterpart for the moment cone (\ref{MF}) for trigonometric functions is obtained using Lemma \ref{lem:B}.
\begin{thm}
Let $t_0\in(-\frac{1}{2},\frac{1}{2}]$ be fixed arbitrarily.
$M\in \mathcal{M}_F$ is uniquely represented with $2n-1$ parameters $(w_0,\ldots,w_{n-1},t_1,\ldots,t_{n-1})$ as
\begin{align*}
 M = \sum_{i=1}^{n-1} w_i f_F(t_i) f_F(t_i)^\top + w_0 f_F(t_0) f_F(t_0), \ \ w_i>0, \\
 t_i\ne t_0, \ \ -\frac{1}{2}<t_1<\cdots<t_{n-1}\le\frac{1}{2}.
\end{align*}
\end{thm}

A square matrix $M=(m_{i,j})$ is said to be Hankel if $m_{i,j}=m_{k,l}$
when $i+j=k+l$.
For example, matrices $M$ in (\ref{M3}) and (\ref{M}) are Hankel.
Obviously, each $M\in \mathcal{M}_P$ should be an $n\times n$ positive definite Hankel matrix.
It is known that the converse is also true.
\begin{prop}
\label{prop:hankel}
The moment cone $\mathcal{M}_P$ in (\ref{MP}) is characterized as
\[
\mathcal{M}_P = \bigl\{ M \succ 0 \mid \mbox{$M$ is Hankel}\,\bigr\}.
\]
\end{prop}
For the proof, see (9.1) of \cite{Karlin-Studden66}, p.\,199. 
This also gives a unique representation of $\mathcal{M}_P$ with $2n-1$ parameters $(m_0,m_1,\ldots,m_{2n-2})$.

Theorem \ref{thm:action} combined with Proposition \ref{prop:hankel} implies that group $\mathcal{A}$ acts on the cone of (positive definite) Hankel matrices.
For the M\"obius group action on Hankel matrices, see also \citet{Heinig-Rost89,Heinig-Rost10}.

\subsection{Invariance under the M\"obius group}

In this subsection, we consider the polynomial regression.
We formalized our optimal experimental design problem to find the minimizer $M\in \mathcal{M}_P$ of $\Vol_1(\gamma_{M^{-1}})$ in (\ref{len}).
\begin{thm}
\label{thm:invariance}
For $M\in \mathcal{M}_P$ and $A\in \mathcal{A}$,
\[
 \Vol_1(\gamma_{M^{-1}}) = \Vol_1\bigl(\gamma_{(A M A^\top)^{-1}}\bigr).
\]
\end{thm}
Theorem \ref{thm:invariance} and Theorem \ref{thm:action} imply
that the minimizer of $\Vol_1(\gamma_{M^{-1}})$ with respect to $M\in \mathcal{M}_P$
forms an orbit (or a union of orbits) on $\mathcal{M}_P$.

\begin{proof}
Let $y=\varphi(x)=\varphi(x;a,b,c,d)=(a x +b)/(c x +d)$.
Then,
$f_P(y) = f_P(\varphi(x)) = A f_P(x) \lambda(x)$,
where $A=A(a,b,c,d)$, $\lambda(x)=\lambda(x;a,b,c,d)$.
Taking derivatives with respect to $x$,
\[
 g_P(y) \dot \varphi(x) = A f_P(x) \dot\lambda(x) + A g_P(x) \lambda(x), \quad
 \dot \varphi(x)
 = \frac{ad-bc}{(cx+d)^2}.
\]
Therefore,
\[
 \bigl( f_P(y), g_P(y) \bigr) =
 A \bigl( f_P(x), g_P(x) \bigr)
 \begin{pmatrix} \lambda(x) & \displaystyle \frac{\dot\lambda(x)}{\dot \varphi(x)} \\ 0 & \displaystyle \frac{\lambda(x)}{\dot \varphi(x)} \end{pmatrix}
\]
and
\begin{align*}
& \det\left(
 \begin{pmatrix} f_P(y)^\top \\ g_P(y)^\top \end{pmatrix} M^{-1}
 \bigl( f_P(y), g_P(y) \bigr) \right)^{\frac{1}{2}} \\
&\quad =
\det\left(
 \begin{pmatrix} f_P(x)^\top \\ g_P(x)^\top \end{pmatrix} (A^{-1}M (A^{-1})^\top)^{-1}
 \bigl( f_P(x), g_P(x) \bigr) \right)^{\frac{1}{2}}
 \frac{\lambda(x)^2}{\dot \varphi(x)}.
\end{align*}
By combining this with
\[
 f_P(y)^\top M^{-1} f_P(y) =
 f_P(x)^\top (A^{-1}M (A^{-1})^\top)^{-1} f_P(x) \lambda(x)^2
\]
and $\dd y=\dot \varphi(x) \dd x$, we have
\begin{align*}
 \Vol_1(\gamma_{M^{-1}})
&=
 2 \int_{-\infty}^{\infty} \frac{\det\left(
 \begin{pmatrix} f_P(y)^\top \\ g_P(y)^\top \end{pmatrix} M^{-1}
 \bigl( f_P(y), g_P(y) \bigr) \right)^{\frac{1}{2}}}
 {f_P(y)^\top M^{-1} f_P(y)} \, \dd y \\
&=
 2 \int_{-\infty}^{\infty} \frac{\det\left(
 \begin{pmatrix} f_P(x)^\top \\ g_P(x)^\top \end{pmatrix} (A^{-1}M (A^{-1})^\top)^{-1}
 \bigl( f_P(x), g_P(x) \bigr) \right)^{\frac{1}{2}}}
 {f_P(x)^\top (A^{-1}M (A^{-1})^\top)^{-1} f_P(x)} \, \dd x \\
&= \Vol_1(\gamma_{A^{-1}M (A^{-1})^\top}).
\end{align*}
\end{proof}

\begin{thm}
\label{thm:invariance-design}
The volumes of the weighted polynomial design
$\left\{\begin{matrix} x_i \\ p_i \end{matrix}\right\}_{1\le i\le N}$ with variance $\sigma_P^2(x;a_0,b_0,c_0,d_0)$, and the design
$\left\{\begin{matrix} y_i \\ p_i \end{matrix}\right\}_{1\le i\le N}$, $y_i=\varphi(x_i;a,b,c,d)$ with variance $\sigma_P^2(y;a',b',c',d')$
are the same, where $(a',b',c',d')$ is determined by
\[
 \varphi(\cdot;a',b',c',d') = \varphi(\cdot;a_0,b_0,c_0,d_0)\circ\varphi^{-1}(\cdot;a,b,c,d).
\]
\end{thm}

\begin{proof}
In (\ref{vw}) of the proof of Theorem \ref{thm:action}, let $w_i=p_i/\sigma_P^2(x_i;a_0,b_0,c_0,d_0)$.
Then, by (\ref{sigma-mobius}),
\[
 v_i = \frac{1}{\lambda(x_i;a,b,c,d)^{2}} \frac{p_i}{\sigma_P^2(x_i;a_0,b_0,c_0,d_0)} = \frac{p_i}{\sigma_P^2(y_i;a',b',c',d')}.
\]
This means that information matrices $M_1$ and $M_2$ of the two designs satisfy $M_1 = A M_2 A^\top$ and hence have the same volume by Theorem \ref{thm:invariance}.
\end{proof}

\subsection{D-optimal design for weighted polynomial regression}
\label{subsec:D-optimal}

We characterize the D-optimal design for the weighted polynomial regression as an orbit of the M\"obius group action.
We start from the fact that in the Fourier regression, the uniform design is D-optimal.
\begin{prop}[\cite{Guest58}]
\label{prop:guest}
In the Fourier regression with the basis (\ref{fourier}),
among the $n$-point discrete design,
only the uniform design
\begin{equation}
\label{uniform-again}
 \left\{\begin{matrix} t^0_i - \theta \\ \frac{1}{n} \end{matrix}\right\}_{1\le i\le n}, \quad t^0_i = \frac{i}{n}-\frac{n+1}{2 n},
\end{equation}
is D-optimal,
where $\theta\in\bigl(-\frac{1}{2n},\frac{1}{2n}\bigr)$ is an arbitrary constant.
The information matrix at the optimal point is the identity $I_n$.
\end{prop}

Let $M_F$ be an information matrix of a Fourier design $\left\{\begin{matrix} t_i \\ p_i \end{matrix}\right\}_{1\le i\le n}$.
By making a change of variables $y_i=\tan(\pi t_i)$ and $y_i=\varphi(x_i;a,b,c,d)$, we have from (\ref{B}), (\ref{A}), and (\ref{sigmaPa0}) that
\begin{align*}
M_F
=& \sum_i f_F(t_i) f_F(t_i)^\top p_i \\
=& B\biggl( \sum_i f_P(y_i) f_P(y_i)^\top \frac{1}{\sigma_P^2(y_i)} p_i \biggr) B^\top \\
=& B A \biggl( \sum_i f_P(x_i) f_P(x_i)^\top \frac{\lambda(x_i;a,b,c,d)^2}{\sigma_P^2(x_i;a,b,c,d)\lambda(x_i;a,b,c,d)^2} p_i \biggr) A^\top B^\top \\
=& B A M_P A^\top B^\top,
\end{align*}
where
\[
 M_P = \sum_i f_P(x_i) f_P(x_i)^\top \frac{1}{\sigma_P^2(x_i;a,b,c,d)} p_i
\]
is the information matrix of the design
$\left\{\begin{matrix} x_i \\ p_i \end{matrix}\right\}_{1\le i\le n}$
for the weighted polynomial regression with variance function $\sigma_P^2(x;a,b,c,d)$.
Because $\det(M_F)=\det(AB)^2\times\det(M_P)$, the D-optimal problem for searching optimal $t_i$ and $p_i$ in the Fourier regression are equivalent to searching for optimal $x_i$ and $p_i$ in the weighted polynomial regression.
Hence, Proposition \ref{prop:guest} is translated into the weighted polynomial regression as follows.

\begin{thm}
\label{thm:d-optimal}
In the weighted polynomial regression of degree $n-1$ with variance $\sigma_P^2(x;a_0,b_0,c_0,d_0)$,
among the $n$-point discrete design, only
the design $\left\{\begin{matrix} x_i \\ \frac{1}{n} \end{matrix}\right\}_{1\le i\le n}$ with
$x_i=\varphi^{-1}(\varphi(\tan(\pi t^0_i);s,-t,t,s);a_0,b_0,c_0,d_0))$
is D-optimal,
where $t^0_i$ is given in (\ref{uniform-again}), and $s,t$ are arbitrary numbers such that $s^2+t^2=1$.
The information matrix at the D-optimal point is $A_0^{-1}(B^\top B)^{-1}(A_0^\top)^{-1}$, where $A_0=A(a_0,b_0,c_0,d_0)$, and $(B^\top B)^{-1}$ is given in (\ref{BB}).
\end{thm}
\begin{proof}
Note that
\[
 y_i = \tan(\pi(t^0_i-\theta)) = \frac{\cos(\pi\theta)\tan(\pi t^0_i)-\sin(\pi\theta)}{\sin(\pi\theta)\tan(\pi t^0_i)+\cos(\pi\theta)}
 = \varphi(\tan(\pi t^0_i);s,-t,t,s),
\]
where $s=\cos(\pi\theta)$, $t=\sin(\pi\theta)$.
$x_i = \varphi^{-1}(y_i;a_0,b_0,c_0,d_0)$.
\end{proof}

When $(a_0,b_0,c_0,d_0)=(1,0,0,1)$, Theorem \ref{thm:d-optimal} reduces to Theorem 3.3 of \cite{Dette-etal99}.

\section{Tube-volume optimal design for $n=3$}
\label{sec:n=3}

In the previous section, we discussed the Fourier regression and the polynomial regression having the basis of (\ref{fourier}) and (\ref{polynomial}), respectively, of a general dimension $n$.
In this section, we treat the case $n=3$.
This is the simplest non-trivial case, because when $n=2$, $\Vol_1(\gamma_{M^{-1}})=2\pi$ irrespective of $M$.

When $n=3$, the problem is reduced to the minimization of
\begin{align*}
 \Vol_1(\gamma_{M^{-1}})
&= 2 \int_{-\infty}^{\infty}
 \frac{\det\left(\begin{pmatrix} 1 & x & x^2 \\ 0 & 1 & 2x \end{pmatrix} M^{-1} \begin{pmatrix} 1 & 0 \\ x & 1 \\ x^2 & 2x \end{pmatrix}\right)^{\frac{1}{2}}}
 {\begin{pmatrix} 1 & x & x^2 \end{pmatrix} M^{-1} \begin{pmatrix} 1 \\ x \\ x^2 \end{pmatrix}} \dd x \\
&= 2 \int_{-\infty}^{\infty} \frac{\sqrt{h_1(x)}}{h_0(x)} \dd x,
\end{align*}
where
\begin{align*}
h_1(x) =&
m_4 (-m_0 m_3^2+m_0 m_2 m_4+2 m_1 m_2 m_3-m_1^2 m_4-m_2^3) \\
& +4 m_3(-m_0 m_2 m_4+m_0 m_3^2-2 m_1 m_2 m_3+m_1^2 m_4+m_2^3) x \\
& +6 m_2(m_0 m_2 m_4-m_0 m_3^2+2 m_1 m_2 m_3-m_1^2 m_4-m_2^3) x^2 \\
& +4 m_1(m_0 m_3^2-m_0 m_4 m_2-2 m_1 m_2 m_3+m_1^2 m_4+m_2^3) x^3 \\
& +m_0(m_0 m_2 m_4-m_0 m_3^2+2 m_1 m_2 m_3-m_1^2 m_4-m_2^3) x^4, \\
h_0(x) =&
m_2 m_4-m_3^2+2(-m_1 m_4+m_2 m_3) x
  +(m_0 m_4+2m_1 m_3-3m_2^2) x^2 \\
& +2(-m_0 m_3+m_1 m_2)x^3+(m_0 m_2-m_1^2) x^4,
\end{align*}
with respect to
\begin{equation}
\label{M}
 M = \begin{pmatrix} m_0 & m_1 & m_2 \\ m_1 & m_2 & m_3 \\m_2 & m_3 & m_4 \end{pmatrix} \succ 0.
\end{equation}
The volume becomes an elliptic integral, which does not have an explicit expression in general.
Moreover, the number of parameters to be optimized is four.
(Note that the integrand $\sqrt{h_1(x)}/h_0(x)$ is a homogeneous function in $m_0,\ldots,m_4$).
We will solve this minimization problem using the M\"obius invariance.

\subsection{Orbital decomposition}

The M\"obius group action defines an equivalence relation on the moment cone $\mathcal{M}_P$.
We define $M_0\sim M_1$ for $M_0, M_1\in \mathcal{M}_P$
if $M_1=A M_0 A^\top$ for some $A\in \mathcal{A}$.
The orbit passing through $M$ is denoted by
\[
 \mathcal{A}(M) = \bigl\{ A M A^\top \mid A \in \mathcal{A} \bigr\} = \{ M_1 \mid M_1 \sim M \}.
\]
The goal of this subsection is to provide the orbital decomposition
of the polynomial moment cone $\mathcal{M}_P$.
Let
\begin{equation}
\label{Mv}
 M_v = \begin{pmatrix} 1 & 0 & v \\ 0 & v & 0 \\ v & 0 & 1 \end{pmatrix}.
\end{equation}

\begin{thm}
\label{thm:orbital}
\[
\mathcal{M}_P
 = \bigsqcup_{v\in(0,\frac{1}{3}]} \mathcal{A}(M_v)
 = \bigsqcup_{v\in[\frac{1}{3},1)} \mathcal{A}(M_v),
\]
where $\sqcup$ is the disjoint union.
\end{thm}
\begin{proof}
This is a consequence of Lemmas \ref{lem:v} and \ref{lem:vv'} below.
\end{proof}

\begin{lem}
\label{lem:v}
For any $M\in \mathcal{M}_P$, there exist $0<v<1$ and $A\in \mathcal{A}$ such that $A M A^\top = M_v$.
\end{lem}

\begin{lem}
\label{lem:vv'}
$M_{v'}\sim M_v$ if and only if $v'=v$ or $v'=(1 - v)/(1 + 3 v)$.
\end{lem}

The proofs of Lemmas \ref{lem:v} and \ref{lem:vv'} are given in Appendix \ref{subsec:v} and \ref{subsec:vv'}, respectively.
Note that the map
\[
 v \mapsto \frac{1 - v}{1 + 3 v}
\]
defines a one-to-one correspondence between $(1,1/3)$ and $(1/3,1)$, and
$v=1/3$ is the fixed point of this map.

The stabilizer of $\mathcal{A}$ at $M\in \mathcal{M}_P$ is defined as
\[
 \mathcal{A}_M = \bigl\{ A\in \mathcal{A} \mid A M A^\top = M \bigr\}.
\]
This is a subgroup of $\mathcal{A}$.

\begin{thm}
\label{thm:stabilizer}
When $v\ne 1/3$,
\[
 \mathcal{A}_{M_v} = \{A(\pm 1,0,0,\pm 1),A(0,\pm 1,\pm 1,0) \} =
 \left\{ \begin{pmatrix} 1 & 0 & 0 \\ 0 & \pm 1 & 0 \\ 0 & 0 & 1 \end{pmatrix},
         \begin{pmatrix} 0 & 0 & 1 \\ 0 & \pm 1 & 0 \\ 1 & 0 & 0 \end{pmatrix}
 \right\},
\]
and when $v=1/3$,
\[
 \mathcal{A}_{M_v}
= \{ A(s,-t,t,s) \mid s^2+t^2=1 \} \sqcup \{ A(-s,t,t,s) \mid s^2+t^2=1 \}.
\]
In particular,
\[
 \dim \mathcal{A}_{M_v} = \begin{cases}
   0 & (v\in (0,\frac{1}{3})\cup (\frac{1}{3},1)), \\
   1 & (v=\frac{1}{3}). \end{cases}
\]
\end{thm}

\begin{proof}
The proof follows from the proof of Lemma \ref{lem:vv'}.
The details are omitted.
\end{proof}

\begin{thm}
\label{thm:orbit}
The dimension of orbit $\mathcal{A}(M_v)$ passing through $M_v$ is
\[
 \dim \mathcal{A}(M_v) = \begin{cases}
   4 & (v\in (0,\frac{1}{3})\cup (\frac{1}{3},1)), \\
   3 & (v=\frac{1}{3}). \end{cases}
\]
\end{thm}
\begin{proof}
Let $M=(m_{i+j-2})_{i,j=1,2,3}$ and
$\widetilde M=(\widetilde m_{i+j-2})_{i,j=1,2,3}$
be $3\times 3$ Hankel matrices, and let $A=A(a,b,c,d)$ ($3\times 3$ matrix in (\ref{A3})).
Assume that $\widetilde M = A M A^\top$.
Picking up the $(1,1)$, $(1,2)$, $(1,3)$, $(2,3)$, and $(3,3)$ elements and rearranging them,
we have $\widetilde m = F m$, where
$m=(m_0,m_1,m_2,m_3,m_4)^\top$,
$\widetilde m=(\widetilde m_0,\widetilde m_1,\widetilde m_2,\widetilde m_3,\widetilde m_4)^\top$,
and
\begin{align*}
 F &= F(a,b,c,d) \\
 &= \begin{pmatrix}
 d^4 & 4 c d^3 & 6 c^2 d^2 & 4 c^3 d & c^4 \\
 b d^3 & a d^3+3 b c d^2 & 3 b d c^2+3 a d^2 c & b c^3+3 a d c^2 & a c^3 \\
 b^2 d^2 & 2 c d b^2+2 a d^2 b & b^2 c^2+4 a b d c+a^2 d^2 & 2 c d a^2+2 b c^2 a & a^2 c^2 \\
 b^3 d & c b^3+3 a d b^2 & 3 b d a^2+3 b^2 c a & d a^3+3 b c a^2 & a^3 c \\
 b^4 & 4 a b^3 & 6 a^2 b^2 & 4 a^3 b & a^4 \\
\end{pmatrix}.
\end{align*}
Let $m_v=(1,0,v,0,1)^\top$.
The tangent space of the orbit at $M_v$ is spanned by the four column vectors of
\[
 \Bigl(\frac{\partial}{\partial a},\frac{\partial}{\partial b},\frac{\partial}{\partial c},\frac{\partial}{\partial d}\Bigr)F(1,0,0,1) m_v
= \begin{pmatrix}
 0 & 0 & 0 & 4 \\
 0 & 1 & 3 v & 0 \\
 2 v & 0 & 0 & 2 v \\
 0 & 3 v & 1 & 0 \\
 4 & 0 & 0 & 0 \\
\end{pmatrix}.
\]
The rank of this matrix is 4 when $v\ne 1/3$ and 3 when $v=1/3$.
\end{proof}

From Theorems \ref{thm:stabilizer} and \ref{thm:orbit}, we see that
$\dim\mathcal{A}_{M_v} + \dim\mathcal{A}(M_v) = 4 = \dim\mathcal{A}$
as expected (e.g., \cite{Kawakubo92}).

\subsection{Minimization over cross-section}

From the orbital decomposition (Theorem \ref{thm:orbital}) and
the invariance of the volume on an orbit (Theorem \ref{thm:invariance}),
the optimization problem is reduced to the minimization of $\Vol_1(\gamma_{M_v^{-1}})$ with respect to $v$, where $M_v$ is defined in (\ref{Mv}).
The range of $v$ is taken to be $(0,\frac{1}{3}]$ or $[\frac{1}{3},1)$.
We write $\Vol_1(\gamma_{M_v^{-1}})=\len(v)$ shortly.
From the definition (\ref{len}),
\[
 \len(v) = \Vol_1\bigl(\gamma_{M_v^{-1}}\bigr) = 2 \int_{-\infty}^\infty s(x;v) \dd x,
\]
where
\[
 s(x;v) = \frac{\sqrt{\frac{1-v^2}{v}}\sqrt{1+6 v x^2+x^4}}{1+(\frac{1}{v}-3 v) x^2+x^4}.
\]
Note that $\len(v)$ is an elliptic integral.
The following is the main theorem of this section.

\begin{thm}
\label{thm:minimizer}
The minimizer of $\Vol_1(\gamma_{M^{-1}})$ in (\ref{len}) over $M\in \mathcal{M}_P$
is given if and only if $M$ is in the orbit
\[
 M \sim M_{1/3}=
 \begin{pmatrix} 1 & 0 & \frac{1}{3} \\ 0 & \frac{1}{3} & 0 \\ \frac{1}{3} & 0 & 1 \end{pmatrix}.
\]
The minimum volume is $4\pi \sqrt{2/3}$.
\end{thm}

\begin{proof}
Because of Theorem \ref{thm:orbital}, it is enough to take the range $v\in (0,\frac{1}{3}]$.
We use the inequality
\begin{equation}
\label{ineq}
 \frac{1}{\sqrt{1+z}} \ge 1-\frac{z}{2}, \quad |z|<1.
\end{equation}
The equality holds iff $z=0$.
Noting that
\[
\frac{1}{\sqrt{1+6 v x^2+x^4}}
 = \frac{1}{(1+x^2) \sqrt{1- \frac{2(1-3v) x^2}{(1+x^2)^2}}}
 \ge \frac{1}{1+x^2} \Bigl( 1+ \frac{(1-3v) x^2}{(1+x^2)^2} \Bigr),
\]
$s(x;v)$ is bounded below by
\[
 \underline s(x;v)
 = \frac{\sqrt{\frac{1-v^2}{v}}(1+6 v x^2+x^4)}{(1+(\frac{1}{v}-3 v) x^2+x^4)(1+x^2)}\Bigl( 1+ \frac{(1-3v) x^2}{(1+x^2)^2} \Bigr).
\]
Therefore, $\len(v)$ is bounded below by
\[
 \underline{\len}(v) = 2 \int_{-\infty}^\infty \underline s(x;v) \dd x.
\]
This integral can be evaluated by counting the residues.
When $v<1/3$, the poles are
\[
 \pm i x_1 = \pm \frac{i}{2} \biggl(\sqrt{\frac{(1-v)(1+3v)}{v}}-\sqrt{\frac{(1+v)(1-3v)}{v}}\biggr),
\]
\[
 \pm i x_2 = \pm \frac{i}{2} \biggl(\sqrt{\frac{(1-v)(1+3v)}{v}}+\sqrt{\frac{(1+v)(1-3v)}{v}}\biggr),
\]
and $\pm i x_0 = \pm i$.

Denote the residues for $+i x_1$, $+i x_2$, and $+i x_0$ by
$\mathrm{Res}(+i x_1)$, $\mathrm{Res}(+i x_2)$, and $\mathrm{Res}(+i x_0)$, respectively.
Then, the integral is evaluated as
\begin{align*}
\underline{\len}(v)
&= 2\times 2\pi i (\mathrm{Res}(+i x_1)+\mathrm{Res}(+i x_2)+\mathrm{Res}(+i x_0)) \\
&=
\frac{2\pi\sqrt{\frac{1-v^2}{v}}
   \Bigl( 3 v^3+6 v^2-5 v+8\sqrt{\frac{v(1+3v)}{1-v}} \Bigr)}{4(1+v)^2}.
\end{align*}
The derivative is
\begin{equation}
\label{derivative}
 \frac{\dd}{\dd v}\underline{\len}(v) =
-2\pi\frac{5 -38 v +14 v^2 +18 v^3 +9 v^4 +48 v \sqrt{\frac{v (1-v)}{1 +3 v}}}
 {8 (1+v)^2 \sqrt{v(1-v^2)}}.
\end{equation}
By applying inequality (\ref{ineq}),
\[
\sqrt{\frac{v (1-v)}{1 +3 v}}
= \frac{\frac{3}{2}v (1-v)}{\sqrt{\frac{9}{4}v (1-v)(1 +3 v)}}
\ge \frac{3}{2}v (1-v) \Bigl\{ 1 - \frac{1}{2}\Bigl(\frac{9}{4}v (1-v)(1 +3 v) -1\Bigr) \Bigr\}
\]
(the equality holds iff $v=1/3$), the numerator of (\ref{derivative}) is
bounded below by
\begin{align*}
& 5 -38 v +14 v^2 +18 v^3 +9 v^4 +48 v \frac{3}{2}v (1-v) \Bigl\{ 1 - \frac{1}{2}\Bigl(\frac{9}{4}v (1-v)(1 +3 v) -1\Bigr) \Bigr\} \\
&\quad = (1-3 v) (5 - 23 v + 53 v^2 - 12 v^3 - 108 v^4 + 81 v^5),
\end{align*}
which is positive for $0<v<1/3$.
Therefore, $\frac{\dd}{\dd v}\underline{\len}(v)<0$ for $0<v<1/3$, and
$\underline{\len}(v)$ has the unique minimum at $v=1/3$.

Since $s(x;v)\ge\underline s(x;v)$,
\[
 \min_{v\in (0,1/3]}\len(v) \ge \min_{v\in (0,1/3]}\underline{\len}(v) = \underline{\len}(1/3).
\]
Moreover, since $s(x;v)=\underline s(x;v)$ at $v=1/3$,
\[
\min_{v\in (0,1/3]}\len(v) \le \len(1/3) = \underline{\len}(1/3).
\]
Therefore,
\[
 \min_{v\in (0,1/3]}\len(v) = \underline{\len}(1/3) = 4\pi \sqrt{2/3}.
\]
Point $v=1/3$ is the unique minimizer, because this is the unique minimizer of $\underline{\len}(v)$.

Figure \ref{fig:len-lowerbound} depicts the objective function $\len(v)$ and its lower bound $\underline{\len}(v)$ for $v\le 1/3$.
\end{proof}

\begin{figure}[H]
\begin{center}
\scalebox{0.6}{\includegraphics{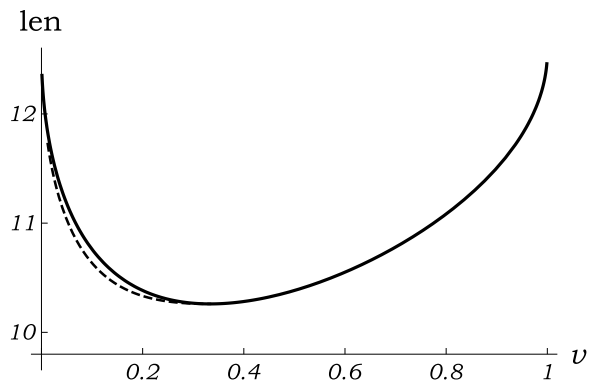}}
\end{center}
\caption{$\len(v)$ (solid line) and its lower bound $\underline{\len}(v)$ for $v\le 1/3$ (dashed line).}
\label{fig:len-lowerbound}
\end{figure}

As shown in (\ref{M3}), the information matrix $M_{1/3}$ is the counterpart of the information matrix for the uniform design in the Fourier regression.

Recall the decomposition of $A(a,b,c,d)$ in Proposition \ref{prop:decomposition}.
We already know from Theorem \ref{thm:stabilizer} that, for $s^2+t^2=1$,
\[
 A(s,-t,t,s) M_{1/3} A(s,-t,t,s)^\top = A(-s,t,t,s) M_{1/3} A(-s,t,t,s)^\top = M_{1/3}.
\]
Moreover,
\[
 A(q,r,0,1) M_{1/3} A(q,r,0,1)^\top
= \begin{pmatrix}
 1 & r & \frac{q^2}{3}+r^2 \\
 r & \frac{q^2}{3}+r^2 & r (q^2+r^2) \\
 \frac{q^2}{3}+r^2 & r (q^2+r^2) & (q^2+r^2)^2
\end{pmatrix}.
\]
Theorem \ref{thm:minimizer} can be written in the following form.
\begin{thm}
The minimizer of $\Vol_1(\gamma_{M^{-1}})$ in (\ref{len}) over $M\in \mathcal{M}_P$
is given when and only when $M$ is of the form:
\[
 M = k
 \begin{pmatrix}
 1 & r & \frac{q^2}{3}+r^2 \\
 r & \frac{q^2}{3}+r^2 & r (q^2+r^2) \\
 \frac{q^2}{3}+r^2 & r (q^2+r^2) & (q^2+r^2)^2
\end{pmatrix}, \quad q\ne 0,\ \ k>0.
\]
\end{thm}

\begin{rem}
\label{rem:circle}
The minimum tube-volume $M\in\mathcal{M}_P$ is attained when and only when the curve
\[
 (\gamma_{M^{-1}})_+ = \bigl\{ \psi_{M^{-1}}(x) = M^{-\frac{1}{2}}f(x)/\Vert M^{-\frac{1}{2}} f(x)\Vert \mid x \in \mathcal{X}=(-\infty,\infty) \bigr\}
\]
forms a circle.
Moreover, in that case, the circle length is $2\pi\sqrt{2/3}$.
\end{rem}

Finally, we characterize the tube-volume optimal design as a three-point design.
The polynomial design corresponding to the Fourier uniform design is given in (\ref{M3}).
The tube-volume optimal design is obtained as an orbit of the transformation passing through the design in (\ref{M3}).
In the following,
let
\begin{equation}
\label{uniform3}
 \left\{\begin{matrix} t^0_i \\ \frac{1}{3} \end{matrix}\right\}_{i=1,2,3}
 =\left\{\begin{matrix} -\frac{1}{3} & 0 & \frac{1}{3} \\
   \frac{1}{3} & \frac{1}{3} & \frac{1}{3} \end{matrix}\right\},
\end{equation}
a three-point uniform design in the Fourier regression.

\begin{thm}
\label{thm:minimum-design-polynomial}
In the weighted polynomial regression with variance function $\sigma_P^2(x;a_0,b_0,c_0,\allowbreak d_0)$, the three-point tube-volume optimal design is
$\left\{\begin{matrix} x_i \\ p_i\end{matrix}\right\}_{i=1,2,3}$,
where
\begin{equation}
\label{xipi}
 x_i = \varphi^{-1}(\tan(\pi t^0_i);a,b,c,d), \quad
 p_i = k\frac{\sigma_P^2(x_i;a_0,b_0,c_0,d_0)}{\sigma_P^2(x_i;a,b,c,d)}.
\end{equation}
Here, $t^0_i$ is defined in (\ref{uniform3}),
$a,b,c,d$ are arbitrarily given so that $a d-b c\ne 0$, and
$k>0$ is a constant so that $\sum p_i=1$.
The tube-volume optimal design includes the D-optimal designs as special cases where
\[
 \varphi(\cdot;a,b,c,d) = \varphi(\cdot;s,-t,t,s)\circ \varphi(\cdot;a_0,b_0,c_0,d_0)
\]
holds for some $s^2+t^2=1$.
\end{thm}

\begin{proof}
The optimal design is the orbit of the M\"obius group passing through the design
$\left\{\begin{matrix} y_i \\ \frac{1}{3} \end{matrix}\right\}_{i=1,2,3}$, $y_i=\tan(\pi t^0_i)$, with variance function $\sigma_P^2(y;1,0,0,1)$.
By Theorem \ref{thm:invariance-design} with $(a',b',c',d')=(1,0,0,1)$,
we see that $\left\{\begin{matrix} x_i \\ \frac{1}{3} \end{matrix}\right\}_{i=1,2,3}$ is an optimal design under variance function $\sigma_P^2(x;a,b,c,d)$.
By Lemma \ref{lem:proportional} with $\sigma_1^2(x)=\sigma_P(x;a,b,c,d)$,
$\sigma_2^2(x)=\sigma_P(x;a_0,b_0,c_0,d_0)$, we have $p_i$ in (\ref{xipi}).
The results for D-optimality is proved in Theorem \ref{thm:d-optimal}.
\end{proof}

\begin{thm}
In the Fourier regression, the three-point tube-volume optimal design is given as
\begin{equation}
\label{tipi}
 \left\{\begin{matrix} t_i \\ p_i\end{matrix}\right\}_{i=1,2,3} =
 \left\{\begin{matrix} \frac{1}{\pi}\tan^{-1}(q \tan(\pi(t^0_i-\theta)) + r) \\
 k \Bigl(\frac{1+(q\tan(\pi(t^0_i-\theta))+r)^2}{1+\tan^2(\pi(t^0_i-\theta))}\Bigr)^2
 \end{matrix}\right\}_{i=1,2,3},
\end{equation}
where $k$ is a normalizing constant so that $\sum_i p_i=1$, $q\ne 0$, and $r$, $\theta$ are arbitrarily given.
In particular, the uniform design (D-optimal design)
$\left\{\begin{matrix} t^0_i-\theta \\ \frac{1}{3} \end{matrix}\right\}_{i=1,2,3}$
is a tube-volume optimal design.
\end{thm}
\begin{proof}
Let $x_i$ and $p_i$ be given in (\ref{xipi}) when $(a_0,b_0,c_0,d_0)=(1,0,0,1)$.
Then, $\sigma_P^2(x;a_0,b_0,c_0,\allowbreak d_0)=\sigma_P^2(x)$, and from Theorem \ref{thm:fourier-polynomial}, the optimal design is given by $\left\{\begin{matrix} t_i \\ p_i \end{matrix}\right\}$, $t_i=\frac{1}{\pi}\tan^{-1}(x_i)$.
Let $(a,b,c,d)$ be chosen such that
$\varphi^{-1}(\cdot;a,b,c,d) = \varphi(\cdot;q,r,0,1)\circ\varphi(\cdot;s,-t,t,s)$ with $s=\cos(\pi\theta)$, $t=\sin(\pi\theta)$.
Then,
$\varphi(\tan(\pi t^0_i);s,-t,t,s)=\tan(\pi(t^0_i-\theta))$ and
$x_i=q\tan(\pi(t^0_i-\theta))+r$, and we have $t_i=\frac{1}{\pi}\tan^{-1}(x_i)$ in (\ref{tipi}).

On the other hand, since $\varphi(\cdot;a,b,c,d)=\varphi(\cdot;s,-t,t,s)\circ \varphi\bigl(\cdot;\frac{1}{q},-\frac{r}{q},0,1\bigr)$,
$\sigma_P^2(x_i;a,b,c,d)\allowbreak =\sigma_P^2\bigl(x_i;\frac{1}{q},-\frac{q}{r},0,1\bigr)=\bigl(1+\frac{1}{q^2}(x_i-r)^2\bigr)^2$.
Therefore,
\[
 p_i = k\frac{\sigma_P^2(x_i;1,0,0,1)}{\sigma_P^2(x_i;a,b,c,d)}
 = k\frac{(1+x_i^2)^2}{\bigl(1+\frac{1}{q^2}(x_i-r)^2\bigr)^2},
\]
which is equivalent to $p_i$ in (\ref{tipi}).
The uniform design corresponds to the case $(q,r)=(1,0)$.
\end{proof}

\subsection{Numerical comparisons}
\label{subsec:numerical}

Here we conduct a small numerical experiment to see the difference of the width of the simultaneous confidence band under optimal and non-optimal designs.
The model we use involves polynomial regression $f(x)=(1,x,x^2)^\top$, $x\in\mathcal{X}=(-\infty,\infty)$, with the variance function of $\sigma^2(x)=(1+x^2)^2$.
The three point designs,
\begin{align*}
& \mathcal{D}(\nu) =\left\{\begin{matrix} -x & 0 & x \\ \frac{p}{2} & 1-p & \frac{p}{2} \end{matrix}\right\} \quad\mbox{with}\ \ x=\frac{1}{\sqrt{v}},\ \ p=\frac{1+v}{2},
\end{align*}
and $v=1/12,1/9,1/6,1/4,1/3,1/2$, are compared.
The information matrix of $\mathcal{D}(\nu)$ is
\[
 M(v) = \frac{1}{2(1+v)}\begin{pmatrix}
 1 & 0 & v \\
 0 & v & 0 \\
 v & 0 & 1
\end{pmatrix},
\]
and the length $\Vol_1(\gamma_{M(v)^{-1}})$ takes its minimum at $v=1/3$, as proved in Theorem \ref{thm:minimizer}.
Table \ref{tab:comparisons} shows the empirical upper $\alpha$ quantiles $w_\alpha$ of the standardized simultaneous confidence bands for the designs $\mathcal{D}(v)$ and their corresponding theoretical values.
We generated the random variable
\[
 \max_{x\in\mathcal{X}}\frac{|(\widehat b-b)^\top f(x)|}{\sqrt{f(x)^\top M(v)^{-1} f(x)}}, \quad \widehat b-b\sim N_3\bigl(0,M(v)^{-1}\bigr),
\]
through simulation with 300,000 replications to obtain the empirical $\alpha$-quantiles
$w_\alpha$.
The corresponding theoretical values by part (i) of Proposition \ref{prop:naiman} are given in parentheses.
As the theorems state, the case $v=1/3$ has the narrowest simultaneous confidence band, although the improvement in the width is not substantial.
\begin{table}[H]
\label{tab:comparisons}
\begin{center}
\caption{Upper $\alpha$ quantiles $w_\alpha$ for the designs $\mathcal{D}(v)$.}
\bigskip
\begin{tabular}{l|c|cc}
\hline
\ $v$ & $\Vol_1(\gamma_{\Sigma})$ & $w_{0.1}$ & $w_{0.05}$ \\
\hline
$1/12$ & 10.872 & 2.3473 \ (2.3879) & 2.6328 \ (2.6624) \\
 $1/9$ & 10.697 & 2.3463 \ (2.3810) & 2.6319 \ (2.6562) \\
 $1/6$ & 10.469 & 2.3438 \ (2.3720) & 2.6283 \ (2.6481) \\
 $1/4$ & 10.304 & 2.3412 \ (2.3653) & 2.6248 \ (2.6421) \\
 $1/3$ & 10.260 & 2.3398 \ (2.3635) & 2.6234 \ (2.6405) \\
 $1/2$ & 10.383 & 2.3411 \ (2.3685) & 2.6251 \ (2.6450) \\
\hline
\end{tabular}

\medskip
(Tube-volume optimal at $v=1/3$. Theoretical values are in parentheses.)
\end{center}
\end{table}

\section{Summary and remaining problems}
\label{sec:summary}

In this paper, we have proposed the tube-volume (TV) criterion $\Vol_1(\gamma_\Sigma)$ in (\ref{len}) in experimental design.
If a design is tube-volume optimal and simultaneously, D-optimal minimizing $\max_{x\in\mathcal{X}} f(x)^\top \Sigma^{-1}f(x)$, the design is optimal that attains the minimum band-width of simultaneous confidence bands.

Then, the proposed criterion was applied to Fourier regression model that is a standard model in linear optimal design theory,
and weighted polynomial regression model that is mathematically equivalent to the Fourier regression model.
The M\"obius group keeps the tube-volume criterion invariant, whereas
the subgroup $\mathrm{O}(2,\R)$ of the M\"obius group keeps the D-criterion invariant.

Using the M\"obius invariance, when $n=3$,
we found that the tube-volume optimal designs in the Fourier regression and the weighted polynomial regression form an orbit of the M\"obius group.
The tube-volume optimal designs contain D-optimal designs as special cases.
This means that in the Fourier regression, the uniform design is a universal optimal design minimizing both tube-volume criterion and D-criterion.

\subsubsection*{A conjecture}

We conjecture that for all $n$, the tube-volume optimal design is characterized as an orbit of the M\"obius group containing D-optimal designs.
One supporting observation is that for small $n$ ($n\le 6$), tube-volume local optimality at the D-optimal designs can be proved by direct calculations.
That is, the Hessian matrix of the tube-volume criterion evaluated at the D-optimal design is positive semi-definite, and
the null space of the Hessian matrix corresponds to the tangent space of the orbit of M\"obius group action.
However, the proof for general $n$ remains outstanding.
As stated in Remark \ref{rem:circle}, when $n=3$, the trajectory of the normalized regression basis of the minimum length is a circle.
Although the trajectory cannot be a circle for $n\ge 4$, some idea of the isoperimetric inequality may be useful.

\subsubsection*{Multivariate extension}

Throughout the paper, we just dealt with the case where the explanatory variable is one-dimensional.
However, the volume-of-tube method works for the construction of the simultaneous confidence bands for regression with multidimensional explanatory variables except for the conservativeness (ii) of Proposition \ref{prop:naiman},
and the volume-optimality is well-defined.
For example, we can discuss the volume-optimality of the $p$-variate polynomial regression model with the basis vector
\begin{align*}
f(x) &= \bigl(1,(x_i)_{1\le i\le p},(x_i x_j)_{1\le i\le j\le p},\ldots,
(x_{i_1} \cdots x_{i_d})_{1\le i_1\le\cdots\le i_d\le p}\bigr)^\top \\
 &\in \R[x_1,\ldots,x_p]^{p+d\choose d}.
\end{align*}
By the same argument as the univariate case, we can prove that
the multivariate M\"obius transform $\varphi:\overline\R^p\to\overline\R^p$
defined by
\[
 x=(x_1,\ldots,x_p)^\top\mapsto \varphi(x;A,b,c,d)=\frac{A x+b}{c^\top x+d}, \quad \det\begin{pmatrix} A & b \\ c^\top & d \end{pmatrix}\ne 0
\]
where $A\in\R^{p\times p}$, $b,\,c\in \R^{p\times 1}$, $d\in\R$
(e.g., \cite{Kato-McCullagh14})
remains the invariance (volume preserving property)
$\Vol_p(\gamma_{M^{-1}})=\Vol_p(\gamma_{(A M A^\top)^{-1}})$
of Theorem \ref{thm:invariance}.
However, the treatment of the multidimensional case (see, e.g., \cite{Lasserre09} for the moment cone) remains a future topic of research.

\subsubsection*{Application to other regression models}

The application of the TV-criterion to other regression models remains an important research topic.
As a simple example, consider the Fourier and the weighted polynomial regressions (\ref{fourier}) and (\ref{polynomial}) whose explanatory domain is a finite interval $\mathcal{X}=[A,B]$.
Then, using arguments similar to those used in Section \ref{sec:n=3}, we can show that the TV-optimal design is an improper two-point design with masses at $A$ and $B$.
This does not coincide with the D-optimal design.

From this observation, we pose two problems:
 (i) To characterize the models in which the proper TV-optimal design exists, and the TV-optimal and D-optimal designs are compatible.
 (ii) How to combine the TV-optimal and D-optimal designs when they are different, for example, a mixture of the D- and TV-optimal designs with appropriate weights.

\appendix
\section{Appendix: Proofs}

\subsection{Proof of Proposition \ref{prop:elliptically}}
\label{subsec:elliptically}

\begin{proof}
Let $\eta=\xi/\Vert\xi\Vert$, where $\xi=\Sigma^{-\frac{1}{2}}(\widehat b-b)$.
$\eta$ is uniformly distributed on the unit sphere $\S^{n-1}$.
Let $\psi(x)=\Sigma^{\frac{1}{2}}f(x)/\Vert\Sigma^{\frac{1}{2}}f(x)\Vert$.
Then, under the assumption that any connected component of $\gamma_\Sigma$ is not a closed curve, 
the upper tail probability of $\max_{x\in\mathcal{X}}|\eta^\top\psi(x)|$ is bounded above by the expectation of the Euler-characteristic $\chi(A_c)$ of the excursion set:
\[
 A_c= \bigl\{ \psi(x) \mid x\in\mathcal{X},\,\eta^\top\psi(x) \ge c \bigr\}
 \cup \bigl\{ -\psi(x) \mid x\in\mathcal{X},\,-\eta^\top\psi(x) \ge c \bigr\}
 \subset\S^{n-1}.
\]
That is, by Proposition 3.2 and (3.10) of \cite{Takemura-Kuriki02},
\begin{align*}
 \Pr\biggl(\max_{x\in\mathcal{X}}|\eta^\top\psi(x)| >c\biggr) \le
 \frac{\Vol_1(\gamma_\Sigma)}{2\pi} \Pr\Bigl(B_{\frac{2}{2},\frac{n-2}{2}}>c^2\Bigr)
 + \chi(\gamma_\Sigma)\Pr\Bigl(B_{\frac{1}{2},\frac{n-1}{2}}>c^2\Bigr), 
\end{align*}
where $B_{\frac{k}{2},\frac{n-k}{2}}$ is a random variable distributed as the beta distribution with parameter $\bigl(\frac{k}{2},\frac{n-k}{2}\bigr)$.

Let $\widetilde\xi$ be a copy of $\xi$ distributed independently of $\xi$ and $B_{\frac{k}{2},\frac{n-k}{2}}$.
Then, we see the equivalence in distribution:
\begin{equation}
 \Vert\widetilde\xi\Vert \times \eta \mathop{=}^d \xi \quad\mbox{and}\quad \Vert\widetilde\xi\Vert^2 \times B_{\frac{k}{2},\frac{n-k}{2}} \mathop{=}^d R_k^2.
\label{equivalence}
\end{equation}
(\ref{equivalence}) can be proved by checking the Mellin transforms.
By letting $c:=c/\Vert\widetilde\xi\Vert$, and taking the expectations with respect to $\widetilde\xi$, we have (\ref{naiman_elliptically}). 
\end{proof}

\subsection{Proof of Lemma \ref{lem:M}}
\label{subsec:gould}

\begin{proof}
Let $t_k=k/n-(n+1)/(2n)$ and $x_k=\tan(\pi t_k)$.
The $(i,j)$ element of $M$ is
\[
 m_{i,j}
= \frac{1}{n}\sum_{k=1}^{n} x_k^{(i-1)+(j-1)}\frac{1}{(1+x_k^2)^{2(n-1)}}
= \frac{1}{n}\sum_{k=1}^{n} \sin^{i+j-2}(\pi t_k) \cos^{n-i-j}(\pi t_k).
\]
Then, apply Lemma \ref{lem:m} below.
\end{proof}

\begin{lem}
\label{lem:m}
Let $t_i= i/n -d$, $i=1,\ldots,n$, where $d$ is a constant.
Then,
\[
\frac{1}{n}\sum_{i=1}^n \sin^{2k}(\pi t_i)\cos^{2n-2k-2}(\pi t_i)
= \frac{\Gamma(k+\frac{1}{2})\Gamma(n-k-\frac{1}{2})}{\pi\Gamma(n)}.
\]
\end{lem}

\begin{proof}
From the duplication formula for the gamma function, it suffices to show that
\[
 \frac{1}{n}\sum_{j=1}^n \sin^{2k}(\pi t_j)\cos^{2n-2k-2}(\pi t_j)
= \frac{(2k)!(2n-2k-2)!}{k!(n-k-1)!(n-1)!}2^{-2n+2}.
\]
Let $\w=e^{i\pi/n}$, $c=e^{-i\pi d}$.
The left-hand side times $(2i)^{2k} 2^{2n-2k-2}=2^{2n-2} (-1)^k$ is
\begin{align}
& \frac{1}{n} \sum_{j=1}^n (c \w^j-c^{-1} w^{-j})^{2k} (c \w^j+c^{-1} w^{-j})^{2n-2k-2} \nonumber \\
&= \frac{1}{n} \sum_{j=1}^n
 \sum_{0\le s\le 2k}{2k \choose s}(-1)^s c^s \w^{sj} c^{-(2k-s)} \w^{-(2k-s)j} \nonumber \\
& \qquad\times
 \sum_{0\le t\le 2n-2k-2}{2n-2k-2 \choose t} c^t \w^{tj} c^{-(2n-2k-2-t)}\w^{-(2n-2k-2-t)j}.
\label{eq1}
\end{align}
Here, the contribution of the summation with respect to $j$ is
\[
 \frac{1}{n} \sum_{j=1}^n \w^{(-2n+2+2s+2t)j} = \begin{cases} 1 & (s+t=n-1), \\ 0 & (\mbox{otherwise}). \end{cases}
\]
Hence, (\mbox{\ref{eq1}}) is equal to
\begin{align*}
\sum_{\max(0,2k-n+1)\le s\le \min(2k,n-1)} (-1)^s {2k \choose s} {2n-2k-2 \choose n-1-s} \\
= (-1)^k \frac{(2k)!(2n-2k-2)!}{k!(n-k-1)!(n-1)!}.
\end{align*}
The last equality follows from (1.41) of \cite{Gould10}.
\end{proof}

\subsection{Proof of Lemma \ref{lem:v}}
\label{subsec:v}

The proof of Lemma \ref{lem:v} is divided into three parts (lemmas).
\begin{lem}
\label{lem:v1}
For $M\in\mathcal{M}_P$, there exist $u,w>1$ such that
\[
 M \sim \begin{pmatrix} u & 1 & 1 \\ 1 & 1 & 1 \\ 1 & 1 & w \end{pmatrix}.
\]
\end{lem}
\begin{lem}
\label{lem:v2}
For $u,w>1$, there exist $\tilde u,\tilde v,\tilde w>0$,
$\tilde v<\sqrt{\tilde u\tilde w}$ such that
\[
 \begin{pmatrix} u & 1 & 1 \\ 1 & 1 & 1 \\ 1 & 1 & w \end{pmatrix} \sim
 \begin{pmatrix} \tilde u & 0 & \tilde v \\ 0 & \tilde v & 0 \\ \tilde v & 0 & \tilde w \end{pmatrix}.
\]
\end{lem}
\begin{lem}
\label{lem:v3}
For $\tilde u,\tilde v,\tilde w>0$,
$\tilde v<\sqrt{\tilde u \tilde w}$, there exist $v\in (0,1)$ such that
\[
 \begin{pmatrix} \tilde u & 0 & \tilde v \\ 0 & \tilde v & 0 \\ \tilde v & 0 & \tilde w \end{pmatrix} \sim M_v.
\]
\end{lem}

\begin{proof}[Proof of Lemma \ref{lem:v1}]
We start from a canonical form in (\ref{KS-MP}):
\[
 M = w_1 f_P(x_1) f_P(x_1)^\top + w_2 f_P(x_2) f_P(x_2)^\top + w_0 f_P(\pm\infty) f_P(\pm\infty)^\top,
\]
where $f_P(x)=(1,x,x^2)^\top$, $f_P(\pm\infty)=(0,0,1)^\top$.
For $A=A(a,b,c,d)$
with $a=w_2^{1/4}(x_1-x_2)$, $b=-w_2^{1/4}x_1$, $c=0$, $d=-w_2^{1/4}$,
we have
\[
A \begin{pmatrix} u & 1 & 1 \\ 1 & 1 & 1 \\ 1 & 1 & w \end{pmatrix} A^\top = M, \quad
u=w_1 d^{-4}+1,\ \ w=w_0 a^{-4}+1.
\]
\end{proof}

\begin{proof}[Proof of Lemma \ref{lem:v2}]
Let
\[
 M = \begin{pmatrix} u & 1 & 1 \\ 1 & 1 & 1 \\ 1 & 1 & v \end{pmatrix}, \quad
 \widetilde M = \begin{pmatrix} \tilde u & 0 & \tilde v \\ 0 & \tilde v & 0 \\ \tilde v & 0 & \tilde w \end{pmatrix}, \quad
 A = \begin{pmatrix} d^2 & 2cd & c^2 \\ bd & bc+ad & ac \\ b^2 & 2ab & a^2 \end{pmatrix}.
\]
We confirm that equation $A M A^\top=\widetilde M$ has a solution $(a,b,c,d)$ such that $ad-bc\ne 0$.
It is enough to show that under the assumption $a,d\ne 0$, a solution $(a,b,c,d)$ satisfies
\begin{equation}
\label{AMAAMA}
 (A M A^\top)_{1,2}=0, \ \ (A M A^\top)_{2,3}=0, \ \ ad-bc\ne 0.
\end{equation}

Solving $(A M A^\top)_{1,2}=0$ with respect to $b$ yields
\begin{equation}
\label{b}
 b= -a \frac{c^3 v+3 c^2 d+3 c d^2+d^3}{d^3 f_1(c/d;u)},
\end{equation}
where
\[
 f_1(c;u) = (c+1)^3 + (u-1)
\]
if $f_1(c/d;u)\ne 0$.
Substituting (\ref{b}) into $(A M A^\top)_{2,3}$, we have
\begin{align}
\label{AMA23}
0= (A M A^\top)_{2,3}=
\frac{a \{(v-1)c^4 + (u-1)d^4 + (c+d)^4\} d^6 f(c/d;u,v) }
{\{ d^3 f_1(c/d;u) \}^3},
\end{align}
where
\begin{align*}
 f(c;u,v)=&
  c^6 (-2 + 3 v - v^2)
+ c^5 (-6 + 7 v - u v^2)
+ c^4 (-10 + 15 v - 5 u v) \\ &
+ c^3 (-10 u + 10 v)
+ c^2 (10 - 15 u + 5 u v)
+ c (6 - 7 u + u^2 v) \\ &
+ (2 - 3 u + u^2).
\end{align*}
Substituting (\ref{b}) into $a d-b c$, we have
\[
 0 \ne ad-bc =
 a \frac{d^4 f_2(c/d;u,v)}{d^3 f_1(c/d;u)}, \quad
 f_2(c;u,v) = 4 c + 6 c^2 + 4 c^3 + u + c^4 v.
\]
In the numerator of (\ref{AMA23}), if $(v-1)c^4 + (u-1)d^4 + (c+d)^4=0$, then $c=d=0$, and hence, $ad-bc=0$.

Now, we examine whether for all $u,v>1$, there is a real $x$ such that
\begin{equation}
\label{fff}
 f(x;u,v)=0,
 \quad f_1(x;u)\ne 0, \quad f_2(x;u)\ne 0.
\end{equation}
Once a solution $x=x^*$ is obtained, we have a solution that $a,d\ne 0$ are given arbitrarily, $c=d x^*$, and $b$ is determined from $(a,c,d)$ in (\ref{b}).

Write $f(\cdot)=f(\cdot;u,v)$ shortly.
It is easily shown that when $u,v>1$, neither
\begin{align*}
f(0)  =& (u-2)(u-1)>0, \\
f(-1) =& -(u-v)(u-1)(v-1)>0, \\
f(\pm\infty) =& -(v-2)(v-1)>0
\end{align*}
nor
\[
f(0)<0,\ f(-1)<0,\ f(\pm\infty)<0
\]
is true, where $f(\pm\infty)=\lim_{c\to\pm\infty}c^{-6}f(c)$.
This means that $f(x;u,v)=0$ has a real solution $x=x^*$.

In order to check $f_1(x^*;u)\ne 0$ and $f_2(x^*;u,v)\ne 0$,
we need to check whether $f$ and $f_1$ (or $f_2$) have a common factor.
For this purpose, we calculate the resultants $R(f,f_1)$ and $R(f,f_2)$:
\[
 R(f,f_1) = (u-2)(u-1) h(u,v)^2, \quad
 R(f,f_2) = h(u,v),
\]
where
\begin{equation}
\label{h}
h(u,v) = -28 + 54 u - 27 u^2 + 54 v
 - 105 u v + 54 u^2 v - 27 v^2 + 54 u v^2 - 30 u^2 v^2 + u^3 v^3.
\end{equation}
(For resultant, see, e.g., \cite{Prasolov04}.
Applications in statistics can be found in \cite{Drton-etal09}.)
As shown in Lemma \ref{lem:h} later,
in the region $u,v>1$, $h(u,n)=0$ iff $u=2$.
Therefore, when $u\ne 2$, we have established that $x=x^*$ satisfying (\ref{fff}) exists, and hence, a solution $(a,b,c,d)$ satisfying (\ref{AMAAMA}) exists.

When $u=2$, (\ref{AMA23}) is reduced to
\[
 0 = \frac{f(c;2,v)}{f_1(c;2)^3} = \frac{(2-v) c (-2 -4 c -3 c^2 -c^3 -c^4 +c^4 v)}{(2+c)^2(1+c+c^2)^3}
\]
and $c=0$ is a solution.
From (\ref{b}), $b=-a/2$, and $a\ne 0$, $d\ne 0$ are arbitrarily given.
In fact, for $A=A(a,-a/2,0,d)$,
\[
 A \begin{pmatrix} 2 & 1 & 1 \\ 1 & 1 & 1 \\ 1 & 1 & v \end{pmatrix} A^\top =
 \begin{pmatrix}
   2 d^4 & 0 & a^2 d^2/2 \\ 0 & a^2 d^2/2 & 0 \\ a^2 d^2/2 & 0 & a^4 (v-7/8)
 \end{pmatrix}.
\]
\end{proof}

\begin{lem}
\label{lem:h}
Let $h(u,v)$ be defined in (\ref{h}).
When $u,v>1$, $h(u,v)\ge 0$, and the equality holds iff $u=v=2$.
\end{lem}
\begin{proof}
Fix $u$ and consider $h(u,v)$ as a function of $v$.
Note that $h(u,1)=(u-1)^3>0$.
\begin{align*}
h_v(u,v)
&= \frac{\partial}{\partial v}h(u,v) \\
&= 54 - 105 u + 54 u^2 + (-54 + 108 u - 60 u^2) v + 3 u^3 v^2.
\end{align*}
$h_v(u,1)=3u(u-1)^2>0$.
It is easy to see that $h_v(u,v)=0$ has a real solution iff $3/2\le u\le 3$.
Therefore, when $u\le 3/2$ or $u\ge 3$, $h_v(u,v)$ is always positive.
Combined with $h(u,1)>0$, this means that $h(u,v)>0$.

Consider the case $3/2<u<3$. The largest zero of $h_v(u,v)$ is
\[
 v_*(u) = \frac{9 - 18 u + 10 u^2 + 3 \sqrt{(u-1)^3(3-u)(2u-3)}}{u^3} > 1.
\]
At this point, the function $h(u,v)$ takes a local minimum
\begin{align*}
 h(u,v_*(u)) = & \frac{27}{u^6} (u-1)^4 \bigl\{(-54 +108 u -72 u^2 +18 u^3 -u^4) \\
& -2(3-u)(2u-3)\sqrt{(u-1)(3-u)(2u-3)} \bigr\} \\
=& 27 (u-1)^4 (u-2)^2 \bigl\{(-54 +108 u -72 u^2 +18 u^3 -u^4) \\
& +2(3-u)(2u-3)\sqrt{(u-1)(3-u)(2u-3)} \bigr\}^{-1}.
\end{align*}
It is easy to check that $-54 +108 u -72 u^2 +18 u^3 -u^4>0$ for $3/2<u<3$.
Hence, $h(u,v_*(u))\ge 0$ for $3/2<u<3$ and the equality holds when $u=2$.
When $u=2$, $v_*(2)=2$.
\end{proof}

\begin{proof}[Proof of Lemma \ref{lem:v3}]
For $A=A(a,b,c,d)$ with $b=c=0$,
\[
A \begin{pmatrix} \tilde u & 0 & \tilde v \\ 0 & \tilde v & 0 \\ \tilde v & 0 & \tilde w \end{pmatrix} A^\top =
 \begin{pmatrix} d^4 \tilde u & 0 & a^2 d^2 \tilde v \\ 0 & a^2 d^2 \tilde v & 0 \\ a^2 d^2 \tilde v & 0 & a^4 \tilde w \end{pmatrix}.
\]
By letting $a=\tilde w^{-1/4}$, $d=\tilde u^{-1/4}$,
$v=a^2 d^2 \tilde v=\tilde v/\sqrt{\tilde u \tilde w}$, we obtain the result.
\end{proof}

\begin{rem}
In the proof of Lemma \ref{lem:v3}, by letting
$d=\tilde u^{-1/4}$,
$a=\tilde v^{-1/2}/\sqrt{3}d=\tilde u^{1/4}/\sqrt{3}\tilde v^{1/2}$,
$w=a^4\tilde w= \tilde u \tilde w/9\tilde v^2$,
we have another representative group element
\[
 \begin{pmatrix} 1 & 0 & \frac{1}{3} \\ 0 & \frac{1}{3} & 0 \\ \frac{1}{3} & 0 & w \end{pmatrix}, \quad w>\frac{1}{9}.
\]
\end{rem}

\subsection{Proof of Lemma \ref{lem:vv'}}
\label{subsec:vv'}

\begin{proof}
The $(1,1)$, $(2,1)$, $(3,2)$, $(3,3)$ components of the equation $M_{v'}=A M_v A^\top$ are
\begin{align*}
(A M_v A^\top)_{1,1} &= c^4 + d^4 + 6 v c^2 d^2, \\
(A M_v A^\top)_{2,1} &= a c^3 + b d^3 + 3 v c d (b c + a d), \\
(A M_v A^\top)_{3,2} &= a^3 c + b^3 d + 3 v a b (b c + a d), \\
(A M_v A^\top)_{3,3} &= a^4 + b^4 + 6 v a^2 b^2,
\end{align*}
respectively.
By solving $(A M_v A^\top)_{2,1}=0$, we have
\[
 b = -\frac{a c (c^2 + 3 d^2 v)}{d (d^2 + 3 c^2 v)}
\]
when $d\ne 0$. ($c$ and $d$ cannot be 0 simultaneously because $ad-bc\ne 0$.)
Suppose first that $d\ne 0$.
Substituting this into $(A M_v A^\top)_{3,2}$, we have
\[
 (A M_v A^\top)_{3,2} = \frac{a^3 c (c^4 - d^4) (c^4 + d^4 + 6 c^2 d^2 v) (-1 +
     9 v^2)}{d^2 (d^2 + 3 c^2 v)^3}.
\]
This becomes 0 when (i) $c=0$, (ii) $c=\pm d$, or (iii) $v=1/3$.
We try to solve the equation $(A M_v A^\top)_{1,1}=(A M_v A^\top)_{3,3}=1$
in each case.
Note that $a=0$ cannot be a solution, because $b$ becomes 0 and $ad-bc=0$.

(i) If $c=0$, then $b$ becomes 0 and $(A M_v A^\top)_{1,1}=d^4$, $(A M_v A^\top)_{3,3}=a^4$.
Hence, $(a,b,c,d)=(\pm 1,0,0,\pm 1)$ (four ways) are the solutions.
In each case,
\[
 v' = (A M_v A^\top)_{3,1} = v.
\]

(ii) If $c=\pm d$, then $b=\mp a$, $(A M_v A^\top)_{1,1}=2c^4(1+3v)=1=(A M_v A^\top)_{3,3}=2a^4(1+3v)$, and hence, $c=\pm a$.
Therefore,
\[
 (a,b,c,d)= 1/(2 + 6 v)^{1/4}\times
 \left\{ \begin{array}{l} (1,-1,1,1), \\ (1,-1,-1,-1), \\ (-1,1,1,1), \\ (-1,1,-1,-1) \end{array}\right.
\]
are solutions.
In each case,
\[
 v' = (A M_v A^\top)_{3,1} = \frac{1 - v}{1 + 3 v}.
\]

(iii) When $v=1/3$ and $d\ne 0$, we have $b=-ac/d$ and
\[
 (A M_v A^\top)_{1,1} = (c^2 + d^2)^2 = 1, \quad
 (A M_v A^\top)_{3,3} = (a^2 + b^2)^2 = \frac{a^4(c^2 + d^2)}{d^4} = \frac{a^4}{d^4} = 1.
\]
Hence, $b=\pm\sqrt{1-a^2}$.
When $d=a$, $c=-b=\mp\sqrt{1-a^2}$.
When $d=-a$, $c=b=\pm\sqrt{1-a^2}$.
In summary,
\begin{align*}
 (a,b,c,d)=
& \bigl(a,\sqrt{1 - a^2},-\sqrt{1 - a^2},a\bigr), \quad
  \bigl(a,-\sqrt{1 - a^2},\sqrt{1 - a^2},a\bigr), \\
& \bigl(a,\sqrt{1 - a^2},\sqrt{1 - a^2},-a\bigr), \quad
  \bigl(a,-\sqrt{1 - a^2},-\sqrt{1 - a^2},-a\bigr).
\end{align*}
In each case,
\[
 v' = (A M_v A^\top)_{3,1} = \frac{1}{3} = v.
\]

(iii') When $v=1/3$ and $d=0$, $(A M_v A^\top)_{2,1}=a c^3=0$. Because $ad-bc\ne 0$, $c\ne 0$, and $a=0$.
In this case, $(A M_v A^\top)_{3,2}=0$,
$(A M_v A^\top)_{1,1}=c^4=1$, $(A M_v A^\top)_{3,3}=b^4=1$.
Hence, $(a,b,c,d)=(0,\pm 1,\pm 1,0)$ (four ways) are the solutions.
In each case,
\[
 v' = (A M_v A^\top)_{3,1} = \frac{1}{3} = v.
\]
\end{proof}

\subsection*{Acknowledgments}
The authors are grateful to Yasuhiro Omori and Tatsuya Kubokawa for their helpful comments.
This work was supported by JSPS KAKENHI Grant Number 16H02792.

\end{document}